\documentclass[10pt]{amsart}

\usepackage{courier}
\usepackage{amsmath,amsthm,amssymb,mathrsfs,amscd}
\usepackage[cmtip,all]{xy}

\usepackage{exscale}

\newtheorem{theorem}{Theorem}[section]
\newtheorem*{theorem*}{Theorem}
\newtheorem{lemma}[theorem]{Lemma}
\newtheorem*{lemma*}{Lemma}
\newtheorem{proposition}[theorem]{Proposition}
\newtheorem*{proposition*}{Proposition}

\newtheorem{corollary}[theorem]{Corollary}

\theoremstyle{definition}
\newtheorem{definition}[theorem]{Definition}
\newtheorem{example}[theorem]{Example}

\theoremstyle{remark}
\newtheorem{remark}[theorem]{Remark}

\numberwithin{equation}{section}

\newcommand{\id}{\mathrm {id}}

\newcommand{\tops}{\mathbf{Top}}

\newcommand{\ktop}{\mathbf{kTop}}
\newcommand{\kotop}{\mathbf{k_\omega Top}}

\newcommand{\fktop}{\mathrm{k}}
\newcommand{\tot}{\mathrm{Tot}}

\DeclareMathOperator{\colim}{colim}
\DeclareMathOperator{\res}{Res}
\DeclareMathOperator{\supp}{supp}

\begin{document}
\title{Cohomology of Local Cochains}
\author[M. Fuchssteiner]{Martin Fuchssteiner}
\email{martin@fuchssteiner.net}

\begin{abstract}
We prove that for generalised partitions of unity $\{ \varphi_i \mid i \in I\}$ 
and coverings 
$\mathfrak{U}:=\{ \varphi_i^{-1} (R \setminus \{ 0 \}) \mid i \in I\}$ of a 
topological space $X$ the cohomology of abstract $\mathfrak{U}$-local 
cochains coincides with the cohomology of continuous $\mathfrak{U}$-local 
cochains, provided the coefficients are loop contractible.  
Furthermore we show that for each locally contractible group $G$ and loop 
contractible coefficient group $V$ the complex of germs of continuous 
functions on left-invariant diagonal neighbourhoods computes 
the Alexander-Spanier and singular cohomology; 
Similar results are obtained for $k$-groups and for 
germs of smooth functions on Lie groups.
\end{abstract}

\maketitle

\section*{Introduction}
It is well known that the Alexander-Spanier cohomology $H_{AS} (X;V)$ of a 
topological space $X$ with coefficients in a real topological vector space $V$ 
coincides with its continuous version, provided the space $X$ is paracompact. 
Analogously for smoothly paracompact manifolds $M$ the Alexander-Spanier 
cohomology $H_{AS} (M;V)$ coincides with the smooth Alexander-Spanier 
cohomology $H_{AS,s} (M;V)$. 
The standard proof thereof uses sheaf theory and is not suited to show that 
one may in fact restrict oneself to diagonal neighbourhoods of a certain kind,
e.g. left-invariant ones in topological groups. In the first part we give an
alternate and more generally applicable proof which also demonstrates that for 
loop contractible coefficients and coverings $\mathfrak{U}$ of a topological 
space $X$ by cozero sets of a generalised partition of unity the cohomology 
of abstract $\mathfrak{U}$-local cochains coincides with the cohomology of 
continuous $\mathfrak{U}$-local cochains, and for manifolds also coincides 
with the cohomology of smooth $\mathfrak{U}$-local cochains, 
if $\mathfrak{U}$ consists of cozero sets of a smooth generalised partition 
of unity and the coefficients are smoothly loop contractible. 
Passage to the colimit over all (numerable) coverings yields the classical 
result for paracompact spaces or smoothly paracompact manifolds. 

We also relate the different cohomology concepts(Alexander-Spanier, singular 
and \v{C}ech cohomology) in locally contractible topological groups and Lie 
groups.  Van Est has already shown in \cite{vE62b} that for locally 
contractible topological groups $G$ on may compute one can compute 
the Alexander-Spanier cohomology by considering left-invariant neighbourhoods 
of the diagonals in $G^{*+1}$ only. We extend his result to continuous and
smooth cochains.

\section{Local, \v{C}ech, Alexander-Spanier and Singular Cochains} 
\label{seccas}

Let $X$ be a topological space and $V$ be an abelian topological group. 
For each open covering $\mathfrak{U}$ of $X$ and each $n \in \mathbb{N}$ 
one can define an open neighbourhood $\mathfrak{U} [n]$ of the diagonal 
in $X^{n+1}$ via
\begin{equation*}
 \mathfrak{U} [n] := \bigcup_{U \in \mathfrak{U}} U^{n+1} \, .
\end{equation*}
These neighbourhoods of the diagonals in $X^{*+1}$ form an open simplicial 
subspace of $X^{*+1}$ which allows us to consider complexes of cochains 
defined on them. We define abelian groups of $n$-cochains and continuous 
$n$-cochains with values in $V$: 
\begin{equation*}
A^n (\mathfrak{U};V) := \left\{ f : \mathfrak{U}[n] \rightarrow V \right\} 
\qquad \text{and} \qquad A_c^n (\mathfrak{U};V) := C (\mathfrak{U} [n] ,V) 
\end{equation*}
The elements of $A^n (\mathfrak{U};V)$ and $A_c^n (\mathfrak{U};V)$ are called 
\emph{$\mathfrak{U}$-local $n$-cochains} and 
\emph{continuous $\mathfrak{U}$-local $n$-cochains} respectively. 
The abelian groups $A^n (\mathfrak{U};V)$ and $A_c^n (\mathfrak{U};V)$ 
form cochain complexes with the usual differential given by 
\begin{equation} \label{stdiff}
df(u_0,\ldots,u_{n+1}):=
\sum_i (-1)^i f(u_0,\ldots,\hat{u}_i,\ldots,u_{n+1}) \, . 
\end{equation} 
The cohomologies of these complexes are denoted by $H (\mathfrak{U};V)$ and 
$H_c (\mathfrak{U};V)$ respectively; they are called the 
\emph{$\mathfrak{U}$-local cohomology} and the 
\emph{continuous $\mathfrak{U}$-local cohomology}. 
The colimit complex $A_{AS}^* (X;V):=\colim_\mathfrak{U}  A^* (\mathfrak{U};V)$
where $\mathfrak{U}$ ranges over all open coverings of $X$ is the complex of 
Alexander-Spanier cochains which computes the Alexander-Spanier cohomology 
$H_{AS} (X;V)$ of $X$. The cohomology of the continuous version 
$A_{AS,c}^* (X;V):=\colim_\mathfrak{U}  A^* (\mathfrak{U};V)$ is the
continuous Alexander-Spanier cohomology $H_{AS,c} (X;V)$ of $X$.

We will show (in Section \ref{secclc}) that the inclusion 
$A_c^* (\mathfrak{U};V) \hookrightarrow A^* (\mathfrak{U};V)$ of cochain
complexes induces an isomorphism in cohomology if $\mathfrak{U}$ is a covering
by cozero sets of a generalised partition of unity. For this purpose we 
consider the \v{C}ech-Alexander-Spanier double complex 
$\check{C}^* ( \mathfrak{U}, A^*)$ for open coverings 
$\mathfrak{U} =\{ U_i \mid \, i \in I \}$ of $X$ whose groups are given by
\begin{equation*}
\check{C}^p ( \mathfrak{U}, A^q) := \left\{ f \in 
\prod_{i_0,\ldots,i_p \in I} A^q (U_{i_0 \ldots i_p};V) \mid \, \forall
\sigma \in S_p : f_{i_0,\ldots,i_p} = \mathrm{sign}
  (\sigma) f_{i_{\sigma (0)} \ldots i_{\sigma (p)} } \right\}
\end{equation*}
and whose horizontal and vertical differentials $d_h$, $d_v$ 
on $\check{C}^p ( \mathfrak{U}, A^q)$ are given by the \v{C}ech coboundary 
operator $\delta$ and $(-1)^p$ times the products of the differentials $d$ 
of the standard complexes $A^* (U_{i_0\ldots i_p};V)$ (cf. Eq. \ref{stdiff}) 
respectively. 
The rows of the double complex 
$\check{C}^p ( \mathfrak{U}, A^q)$ can be augmented by the complex 
$A^* (\mathfrak{U};V)$ of $\mathfrak{U}$-local cochains and the columns can 
be augmented by the \v{C}ech-complex $\check{C}^* (\mathfrak{U};V)$ for the 
covering $\mathfrak{U}$:

\begin{equation*}
  \vcenter{
  \xymatrix{
\vdots & \vdots & \vdots & \vdots \\ 
A^2 (\mathfrak{U};V) \ar[r] \ar[u]_{d_{v}} 
& \check{C}^0 ( \mathfrak{U}, A^2) \ar[r]^{d_{h}}\ar[u]_{d_{v}} 
&  \check{C}^1 ( \mathfrak{U}, A^2) \ar[r]^{d_{h}}\ar[u]_{d_{v}} 
& \check{C}^2 ( \mathfrak{U}, A^2) \ar[r]^{d_{h}}\ar[u]_{d_{v}} & \cdots \\
A^1 (\mathfrak{U};V) \ar[r] \ar[u]_{d_{v}} 
& \check{C}^0 ( \mathfrak{U}, A^1) \ar[r]^{d_{h}}\ar[u]_{d_{v}} 
& \check{C}^1 ( \mathfrak{U}, A^1) \ar[r]^{d_{h}}\ar[u]_{d_{v}} 
& \check{C}^2 ( \mathfrak{U}, A^1) \ar[r]^{d_{h}}\ar[u]_{d_{v}} 
& \cdots \\
A^0 (\mathfrak{U} ;V) \ar[r] \ar[u]_{d_{v}} 
& \check{C}^0 ( \mathfrak{U}, A^0) \ar[r]^{d_{h}}\ar[u]_{d_{v}} 
& \check{C}^1 ( \mathfrak{U}, A^0) \ar[r]^{d_{h}}\ar[u]_{d_{v}} 
& \check{C}^2 ( \mathfrak{U}, A^0) \ar[r]^{d_{h}}\ar[u]_{d_{v}} & \cdots \\
& \check{C}^0 (\mathfrak{U};V) \ar[r]^{d_{h}}\ar[u] 
& \check{C}^1 (\mathfrak{U};V) \ar[r]^{d_{h}}\ar[u] 
& \check{C}^2 (\mathfrak{U};V) \ar[r]^{d_{h}} \ar[u] 
& \cdots
}}
\end{equation*}
We denote the total complex of the double complex 
$\check{C}^* (\mathfrak{U}, A^*)$ by $\tot \check{C}^* (\mathfrak{U}, A^*)$. 
The augmentations of the rows and columns of the double complex 
$\check{C}^* (\mathfrak{U}, A^*)$ induce homomorphisms 
$i^*:A^* (\mathfrak{U};V) \rightarrow \tot \check{C}^* (\mathfrak{U}, A^*)$ and 
$j^*:\check{C}^* (\mathfrak{U};V) \rightarrow 
\tot \check{C}^* (\mathfrak{U}, A^*)$ 
of cochain complexes respectively. 
\begin{lemma} \label{columnsexact}
  The homomorphism $j^*:\check{C}^* (\mathfrak{U};V) \rightarrow 
\tot \check{C}^* (\mathfrak{U}, A^*)$ induces an isomorphism in cohomology.
\end{lemma}

\begin{proof}
The augmented columns of the double complex 
$\tot \check{C}^* (\mathfrak{U}, A^*)$ are exact, because the standard complex 
$A^* (U_{i_0 \ldots i_q};V)$ of a topological space $U_{i_0 \ldots i_q}$ is
always exact. Therefore the augmentation $j^*$ induces an isomorphism in 
cohomology.
\end{proof}

The augmented rows 
$A^q (\mathfrak{U};A) \hookrightarrow \check{C}^* ( \mathfrak{U}, A^q)$ 
are also exact; in fact, extending every function 
$f_{i_0 \ldots i_p} \in A^q (U_{i_0 \ldots i_p} ;V)$
to $U_{i_1 \ldots i_p}^q$ by requiring it to be zero outside 
$U_{i_0 \ldots i_p}^q$ we observe:

\begin{proposition} \label{propchechcontrhom}
For any point finite set $\{ \varphi_{q,i} \mid i \in I\}$ of (not necessarily 
continuous) $\mathbb{Z}$-valued functions on $\mathfrak{U}[q]$ satisfying 
$\sum_i \varphi_i =1$ and 
${\varphi_i }_{\mid \mathfrak{U}[q] \setminus U_i^{q+1}}=0$ the homomorphisms 
\begin{equation} \label{eqchechcontrhom}
  h^{p,q} : \check{C}^p ( \mathfrak{U}, A^q) \rightarrow 
\check{C}^{p-1} ( \mathfrak{U}, A^q), \quad
h^{p,q} ( f )_{i_0 \ldots i_{p-1}} = 
\sum_{i} \varphi_{q,i} \cdot f_{i i_0 \ldots i_{p-1}}
\end{equation}
form a row contraction of the augmented row 
$A^q (\mathfrak{U};A) \hookrightarrow \check{C}^* ( \mathfrak{U}, A^q)$.
\end{proposition}

\begin{proof}
This is similar to the row contractions of the \v{C}ech-deRham complex of a 
manifold using smooth partitions of unity $\varphi_i$ subordinate to 
$\mathfrak{U}$. 
For any cochain $f \in \check{C}^p (\mathfrak{U};A^q)$ of bidegree $(p,q)$ the
horizontal coboundary of $h^{p,q} (f)$ computes to 
\begin{eqnarray*}
  (\delta h^{p,q} (f))_{i_0 \ldots i_p} (\vec{x}) & = & 
\sum_{k=0}^p (-1)^k h^{p,q} (f)_{i_0 \ldots \hat{i}_k \ldots i_p} (\vec{x}) \\
& = & \sum_{k=0}^p (-1)^k \sum_{i} 
\varphi_{q,i} (\vec{x}) f_{i i_0 \ldots \hat{i}_k \ldots i_{p-1}} (\vec{x})\\
& = &  \sum_{i} \varphi_{q,i} (\vec{x}) \sum_{k=0}^p (-1)^k 
f_{i i_0 \ldots \hat{i}_k \ldots i_{p-1}} (\vec{x})\\
& = & \sum_{i} \varphi_{q,i} (\vec{x}) \left[ f_{i_0 \ldots i_p} (\vec{x}) - 
(\delta f)_{i i_0 \ldots i_p} (\vec{x}) \right] \\
& = & f_{i_0 \ldots i_p} (\vec{x}) - 
h^{p+1,q} (\delta f)_{i i_0 \ldots i_p} (\vec{x}) \, ,
\end{eqnarray*}
hence $h^{*,q}$ is a row contraction of the augmented row 
$A^q (\mathfrak{U};A) \hookrightarrow \check{C}^* ( \mathfrak{U}, A^q)$.
\end{proof}

\begin{corollary} \label{rowsalwexact}
  For any open covering $\mathfrak{U}=\{ U_i \mid i \in I \}$ 
of a topological space $X$ the homomorphism 
$i^*:A^* (\mathfrak{U};V) \rightarrow \tot \check{C}^* (\mathfrak{U}, A^*)$
induces an isomorphism in cohomology.
\end{corollary}

\begin{proof}
This follows from the fact that sets $\{ \varphi_{q,i} \}$ of functions as 
required in Proposition \ref{propchechcontrhom} always exist: 
Well order the index set $I$ and inductively define the functions 
$\varphi_{q,i}$ to be 
$\varphi_{q,i} :=1 -\max \{ \varphi_j \mid j < i \}$ on $U_i^{q+1}$ and zero on 
$X \setminus U_i$.
\end{proof}

\begin{corollary} \label{cechisolocal}
  For any open covering $\mathfrak{U}$ of a topological space $X$ the 
\v{C}ech cohomology $\check{H} (\mathfrak{U};V)$ for the covering
$\mathfrak{U}$ and the cohomology $H (\mathfrak{U};V)$ of $\mathfrak{U}$-local
cochains are isomorphic.Thus the \v{C}ech Cohomology 
$\check{H} (\mathfrak{U};V)$ can be computed from the complex 
$A^* (\mathfrak{U};V)$ of $\mathfrak{U}$-local cochains.
\end{corollary}

\begin{example}
  If $U$ is an open identity neighbourhood of a topological group $G$ then the
  open covering 
  $\mathfrak{U}_U:= \{ gU \mid \, g \in G \}$ of $G$ is left invariant and 
the \v{C}ech Cohomology for the covering $\mathfrak{U}_U$ can be computed 
from the complex $A^* (\mathfrak{U};V)$.
\end{example}

Passing to the colimit over all open coverings or all numerable open coverings 
yields the classical result: 
\begin{corollary}[Well known]
  For any topological space $X$ the \v{C}ech cohomology 
$\check{H} (X;V)$ and the Alexander-Spanier cohomology 
$H_{AS} (X;V)$ are isomorphic. The same is true for the \v{C}ech cohomology 
w.r.t. numerable coverings and the Alexander-Spanier cohomology
w.r.t. numerable coverings.
\end{corollary}

If we replace the pre-sheaf $A^q ( -;V)$ by the pre-sheaf $S^q ( -;V)$ of
singular $q$-cochains we obtain a double complex 
$\check{C}^p ( \mathfrak{U}, S^q)$ for every open cover $\mathfrak{U}$ of
$X$. The rows of this double complex can be augmented by the complex 
$S^* (\mathfrak{U};V)$ of cochains on $\mathfrak{U}$-small singular
simplices and the columns can be augmented by the \v{C}ech complex 
$\check{C}^p ( \mathfrak{U}, V)$. However the augmented columns need not be
exact; this only happens if each open set $U_{i_0 \ldots i_p}$ is $V$-acyclic,
(i.e. it has trivial reduced singular cohomology with coefficients $V$), 
e.g. if each open set $U_{i_0 \ldots i_p}$ is contractible. 

\begin{lemma}
 For any open covering $\mathfrak{U}=\{ U_i \mid i \in I\}$ of $X$ for which 
the sets $U_{i_0 \ldots i_p}$ are $V$-acyclic 
the \v{C}ech cohomology $\check{H} (\mathfrak{U};V)$ for the covering 
$\mathfrak{U}$ and the singular cohomology $H_{sing} (X;V)$ are isomorphic. 
In particular the \v{C}ech cohomology does not depend on the open cover
subject to the acyclicity condition chosen. 
\end{lemma}

\begin{proof}
If each reduced singular cohomology $\tilde{H}_{sing} (U_{i_0 \ldots i_p};V)$ 
is trivial then the augmented columns of the double complex 
$\check{C}^p ( \mathfrak{U}, S^q)$ are exact. Proceeding as for local cochains 
yields an isomorphism 
$H ( S^* (\mathfrak{U};V)) \cong \check{H} (\mathfrak{U};V)$. 
The cohomology of the complex $S^* (\mathfrak{U};V)$ is the singular 
cohomology $H_{sing} (X;V)$ of $X$.
\end{proof}

\begin{example}
  If $\mathfrak{U}$ is a 'good' cover of a topological space $X$ then 
 the \v{C}ech cohomology $\check{H} (\mathfrak{U};V)$ for the covering 
$\mathfrak{U}$ is isomorphic to the singular cohomology $H_{sing} (X;V)$ of $X$.
\end{example}

\begin{example}
If $\mathfrak{U}$ is an open covering of a finite dimensional Riemannian 
manifold $M$ by geodetically convex sets, then the \v{C}ech cohomology 
$\check{H} (\mathfrak{U};V)$ for the covering $\mathfrak{U}$ is isomorphic to 
the singular cohomology $H_{sing} (M;V)$ of $M$. 
If $M$ is an infinite dimensional Riemannian manifold one has to require the 
local existence of geodesics for this argument to be applicable.
\end{example}

\begin{example}
  If $G$ is a Hilbert Lie group (i.e. a Lie group whose model space is a
  Hilbert space\footnote{We do not restrict ourselves to finite dimensional Lie groups.}), and $U$ a geodetically convex identity neighbourhood of $G$, 
then the the \v{C}ech cohomology $\check{H} (\mathfrak{U}_U;V)$ for the 
covering $\mathfrak{U}_U:=\{ gU \mid g \in G \}$ is isomorphic to the singular 
cohomology $H_{sing} (G;V)$ of $G$.
\end{example}

In this case the singular and $\mathfrak{U}$-local cohomologies also
coincide. We assert that this isomorphism is induced by a natural morphism 
$\lambda^* : A^* (\mathfrak{U};V) \rightarrow S^* (X,\mathfrak{U};V)$ of
cochain complexes whose construction we briefly recall:
Consider the singular semi-simplicial space $C (\varDelta , X)$ of $X$ and 
the vertex morphism $\lambda_X : C (\varDelta , X ) \rightarrow X^{*+1}$ of 
semi-simplicial spaces, which assigns to each singular $n$-simplex 
$\tau : \varDelta^n \rightarrow X$ its ordered set of vertices 
$(\tau ( \vec{e}_0) ,\ldots, \tau (\vec{e}_n))$. This morphism $\lambda$
induces a morphism 
$C ( \lambda_\mathfrak{U}^*,V) : A^* (\mathfrak{U};V) \rightarrow 
S^* (X,\mathfrak{U};V)$ of cochain complexes.

\begin{proposition}
For any open covering $\mathfrak{U}=\{ U_i \mid i \in I\}$ of $X$ for which 
the sets $U_{i_0 \ldots i_p}$ are $V$-acyclic the morphism 
$C (\lambda_\mathfrak{U}^* ,V) : 
A^* (\mathfrak{U};V) \rightarrow S^* (X,\mathfrak{U};V)$ 
induces an isomorphism $H (\mathfrak{U};V) \cong H_{sing} (X;V)$ in
cohomology. In particular the \v{C}ech and $\mathfrak{U}$-local cohomology do 
not depend on the open cover subject to the acyclicity condition chosen.
\end{proposition}

\begin{proof}
Let $\mathfrak{U}$ be an open covering of $X$ satisfying 
$\tilde{H}_{sing} (U_{i_0 \ldots i_p};V)=0$ for all $i_0 \ldots i_p \in I$. 
The morphism $\lambda_X : C (\varDelta , X ) \rightarrow X^{*+1}$ of 
semi-simplicial not only induces a morphism 
$C (\lambda_\mathfrak{U}^* ,V) : 
A^* (\mathfrak{U};V) \rightarrow S^* (X,\mathfrak{U};V)$ of cochain complexes
but also a morphism 
$\check{C}^* (\mathfrak{U};C (\lambda^*;V) ) : \check{C}^* (\mathfrak{U};A^* )
\rightarrow \check{C}^* (\mathfrak{U};S^* )$ 
of double complexes. The morphisms $C (\lambda_\mathfrak{U}^* ,V)$ and 
$\check{C}^* (\mathfrak{U};C (\lambda^*;V) )$ intertwine the augmentations of
the double complexes $\check{C}^* (\mathfrak{U};A^* )$ and 
$\check{C}^* (\mathfrak{U};S^* )$, leading to the commutative diagram
\begin{equation*}
  \xymatrix{
 A^* (\mathfrak{U};V) \ar[r]^{i^*} \ar[d]_{C (\lambda_\mathfrak{U}^* ;V)} & 
\tot \check{C}^* ( \mathfrak{U}, A^*) 
\ar[d]|{\check{C}^* (\mathfrak{U};C (\lambda^*;V) )} & 
\check{C} (\mathfrak{U};V) \ar[l]_{j^*} \ar@{=}[d] \\
S^* (\mathfrak{U};V) \ar[r]^{i^*} & \tot \check{C}^* ( \mathfrak{U}, S^*) & 
\check{C} (\mathfrak{U};V) \ar[l]_{j^*} 
}
\end{equation*}
of cochain complexes, in which all but the left downward morphisms induce 
isomorphisms in cohomology. Thus the morphism 
$C (\lambda_\mathfrak{U}^* ,V) : 
A^* (\mathfrak{U};V) \rightarrow S^* (X,\mathfrak{U};V)$ 
of cochain complexes induces an isomorphism in cohomology as well.
\end{proof}

\begin{corollary}
For any open covering $\mathfrak{U}=\{ U_i \mid i \in I\}$ of $X$ for which 
the sets $U_{i_0 \ldots i_p}$ are $V$-acyclic the singular cohomology of $X$ 
and the \v{C}ech cohomology for the covering $\mathfrak{U}$ can be computed 
from the complex $A^* (\mathfrak{U};V)$ of $\mathfrak{U}$-local cochains.
\end{corollary}

\begin{example}
  For any good covering $\mathfrak{U}$ of a topological space $X$ the morphism 
$C (\lambda_\mathfrak{U}^* ,V) : 
A^* (\mathfrak{U};V) \rightarrow S^* (X,\mathfrak{U};V)$ of cochain complexes
induces an isomorphism in cohomology and the cohomologies 
$\check{H} (\mathfrak{U};V)$, $H (\mathfrak{U};V)$ and $H_{sing} (X;V)$ are
isomorphic. 
\end{example}

\begin{lemma}
If the open coverings $\mathfrak{U}$ a topological space $X$ for which 
the sets $U_{i_0 \ldots i_p}$ are $V$-acyclic 
are cofinal in all open coverings, then for each such
covering $\mathfrak{U}$ the \v{C}ech cohomology $\check{H} (\mathfrak{U};V)$
for the covering $\mathfrak{U}$ coincides with the 
\v{C}ech cohomology $\check{H} (X;V)$ of $X$ and the cohomology 
$H (\mathfrak{U};V)$ of $\mathfrak{U}$-local cochains coincides with the
Alexander-Spanier cohomology $H_{AS} (X;V)$ of $X$. In particular the systems 
$H (\mathfrak{U};V)$ and $\check{H} (\mathfrak{U};V)$ of abelian groups are 
co-Mittag-Leffler.
\end{lemma}

\begin{proof} 
If the open coverings $\mathfrak{U}$ satisfying 
$\tilde{H}_{sing} (U_{i_0 \ldots i_p};V)=0$ for all $i_0 \ldots i_p \in I$ of
a space $X$ are cofinal in all open coverings, then the 
Alexander-Spanier and \v{C}ech cohomologies of $X$ can be computed as the 
colimit over all such open covers $\mathfrak{U}$, hence 
$H_{AS} (X;V) = \colim_\mathfrak{U} H (\mathfrak{U};V) 
\cong \colim_\mathfrak{U} H_{sing} (X;V) =  H_{sing} (X;V)$   
and  
$\check{H} (X;V) = \colim_\mathfrak{U} \check{H} (\mathfrak{U};V) \cong 
\colim_\mathfrak{U} H_{sing} (X;V) =  H_{sing} (X;V)$, where the colimits are
taken over all open coverings $\mathfrak{U}$ subject to the condition 
$\tilde{H}_{sing} (U_{i_0 \ldots i_p};V)=0$ for all $i_0 \ldots i_p \in I$.
\end{proof}

\begin{example} \label{exugeodconv}
Every covering of a finite dimensional Riemannian manifold $M$ admits an open 
refinement $\mathfrak{U}$ by geodetically convex sets. Thus for any open
covering $\mathfrak{U}$ of a finite dimensional Riemannian manifold $M$ by 
geodetically convex sets the cohomology of the complexes  
$A^* (\mathfrak{U};V)$ and $\check{C}^* (\mathfrak{U};V)$ are isomorphic to 
the Alexander-Spanier, \v{C}ech and singular cohomologies of $M$. 
If $M$ is an infinite dimensional Riemannian manifold one has to require the 
local existence of geodesics for this argument to be applicable.
\end{example}

\begin{example} \label{exghlg}
  If $G$ is a Hilbert Lie group and $U$ a geodetically convex identity 
neighbourhood of $G$, then then the cohomology of the complexes 
$A^* (\mathfrak{U}_U;V)$  and $\check{C}^* (\mathfrak{U};V)$ are both isomorphic
to $H_{AS} (G;V)$, $\check{H} (\mathfrak{U};V)$ and $H_{sing} (G;V)$.
\end{example}

For some topological spaces the acyclicity condition on the covering 
$\mathfrak{U}$ is not necessary to obtain similar results. For topological 
groups one can also consider the complexes 
$\colim_{U \in \mathcal{U}_1} A^* (\mathfrak{U}_U ;V)$ and 
$\colim_{U \in \mathcal{U}_1} A_c^* (\mathfrak{U}_U ;V)$, where $U$ ranges
over all open identity neighbourhoods of $G$. For these colimit complexes we 
observe:

\begin{theorem} \label{colimnbhf}
For any locally contractible topological group $G$ with open identity
neighbourhood filterbase $\mathcal{U}_1$ the morphisms 
$C (\lambda_{\mathfrak{U}_U}^* ,V) : 
A^* (\mathfrak{U}_U ; V) \rightarrow S^* (\mathfrak{U};V)$ induce an
isomorphism 
$\colim_{U \in \mathcal{U}_1} H (\mathfrak{U}_U ;V) \cong H_{sing} (X;V)$
in cohomology and the cohomologies 
$\colim_{U \in \mathcal{U}_1} H (\mathfrak{U}_U ;V)$, 
$\colim_{U \in \mathcal{U}_1} \check{H} (\mathfrak{U}_U ;V)$, and 
$H_{sing} (G;V)$ coincide.
\end{theorem}

\begin{proof}
The first statement has been shown by van Est in \cite{vE62b}.
The other isomorphisms then are a consequence of Corollary \ref{cechisolocal}.
\end{proof}

\begin{corollary}
  For Lie groups $G$ with open identity neighbourhood filterbase 
$\mathcal{U}_1$ the cohomologies 
$\colim_{U \in \mathcal{U}_1} H (\mathfrak{U}_U ;V)$, 
$\colim_{U \in \mathcal{U}_1} \check{H} (\mathfrak{U}_U ;V)$ and 
$H_{sing} (G;V)$ are isomorphic and the isomorphism 
$\colim_{U \in \mathcal{U}_1} H_c (\mathfrak{U}_U ;V) \cong H_{sing} (G;V)$ is
induced by the vertex morphism $\lambda$.
\end{corollary}

\section{Continuous Local Cochains}
\label{secclc}

The preceding observations can -- in parts -- be generalised for continuous
cochains. 
Replacing the pre-sheaf $A^q (-;V)$ of $q$-cochains by the pre-sheaf 
$A_c^q (-;V)$ of continuous $q$-cochains we obtain a sub double complex 
$\check{C}^* ( \mathfrak{U}, A_c^*)$ of $\check{C}^* ( \mathfrak{U}, A^*)$ 
whose groups are given by
\begin{equation*}
\check{C}^p ( \mathfrak{U}, A_c^q) :=
\left\{ f \in \check{C}^p ( \mathfrak{U}, A^q) \mid \forall i_0 , \ldots i_p
  \in I: f_{i_0 \ldots i_p} \in C ( U_{i_0 \ldots i_p}^{q+1} ;V ) \right\} \, .
\end{equation*}
The rows of this sub double complex $\check{C}^* ( \mathfrak{U}, A_c^*)$ 
can be augmented by the complex $A_c^* (\mathfrak{U};V)$ of continuous 
$\mathfrak{U}$-local cochains and the columns can be augmented by the 
\v{C}ech-complex $\check{C}^* (\mathfrak{U};V)$ for the covering 
$\mathfrak{U}$.
These augmentations induce homomorphisms 
$i_c^*:A_c^* (\mathfrak{U};V) \rightarrow 
\tot \check{C}^* (\mathfrak{U},A_c^*)$ 
and 
$j_c^*:\check{C}^* (\mathfrak{U};V) \rightarrow 
\tot \check{C}^* (\mathfrak{U}, A_c^*)$ 
of cochain complexes respectively.

\begin{lemma} \label{jindiso2}
  The homomorphism $j_c^*:\check{C}^* (\mathfrak{U};V) \rightarrow 
\tot \check{C}^* (\mathfrak{U}, A_c^*)$ induces an isomorphism in cohomology.
\end{lemma}

\begin{proof}
  The proof is analogous to that of Lemma \ref{columnsexact}.
\end{proof}

To obtain exact rows in the double complex $\check{C}^* (\mathfrak{U}, A_c^*)$ 
we impose a restriction on the coefficient group $V$. For modules $V$ over 
some unital topological ring $R$ one can replace the set 
$\{ \varphi_{q,i} \}$ of functions in Proposition \ref{propchechcontrhom} by 
generalised partitions of unity to obtain results analogous to those in 
Section \ref{seccas}, as will be shown below. In the following the
coefficients $V$ will always
be a topological module  over a unital topological ring $R$.
(Different coefficient groups are considered in the next section.)

\begin{definition}
A \emph{generalised $R$-valued partition of unity} 
$\{ \varphi_i \mid i \in I\}$ on a topological space $X$ is a set of
continuous functions 
$\varphi_i : X \rightarrow R$ which satisfies the equation 
$\sum_{i \in I} \varphi_i =1$. 
An \emph{$R$-valued partition of unity} $\{ \varphi_i \mid i \in I\}$ on a 
topological space $X$ is a generalised $R$-valued partition of unity whose 
supports $\supp \varphi_i = \overline{\varphi_i^{-1} (R \setminus \{ 0 \})}$ 
form a locally finite covering of $X$.
\end{definition}

\begin{remark}
For $R=\mathbb{R}$ one can always replace the functions 
$\varphi_i$ by positive ones with the same zero set 
(cf. Lemma \ref{sumvaprhicontthensumabsvalcont}).
\end{remark}

For any $R$-valued partition of unity $\{\varphi_{q,i} \mid i \in I \}$ 
subordinate to the open covering $\{U_i^{q+1} \mid i \in I \}$ of 
$\mathfrak{U}[q]$ and continuous cochain 
$f \in \check{C}^p (\mathfrak{U};A_c^q)$  
the products $\varphi_{q,i} f_{i i_0 \ldots i_{p-1}}$ have supports in the
open sets $U_{i i_0 \ldots i_p}^{q+1}$ respectively. Therefore they can be
continuously extended to $U_{i_0 \ldots i_{p-1} }^{q+1}$ by defining it to be 
zero outside $U_{i i_0 \ldots i_{p-1}}^{q+1}$. 
Understanding each function $\varphi_{q,i} f_{i i_0 \ldots i_{p-q } }$ to be
extended this way we define a homotopy operator for continuous cochains:

\begin{proposition} \label{propchechcontrhomcont}
For any $R$-valued partition of unity $\{ \varphi_{q,i} \mid i \in I\}$ 
subordinate to the covering $\{U_i^{q+1} \mid i \in I \}$ of $\mathfrak{U}[q]$ 
the homomorphisms 
\begin{equation} \label{eqchechcontrhomcont}
  h^{p,q} : \check{C}^p ( \mathfrak{U}, A^q) \rightarrow 
\check{C}^{p-1} ( \mathfrak{U}, A^q), \quad
h^{p,q} ( f )_{i_0 \ldots i_{p-1}} = 
\sum_{i} \varphi_{q,i} \cdot f_{i i_0 \ldots i_{p-1}}
\end{equation}
form a row contraction of the augmented row 
$A^q (\mathfrak{U};V) \hookrightarrow \check{C}^* ( \mathfrak{U}, A^q)$ which
restricts to a row contraction of the augmented sub-row 
$A_c^q (\mathfrak{U};V) \hookrightarrow \check{C}^* ( \mathfrak{U}, A_c^q)$.
\end{proposition}

\begin{proof}
  By Proposition \ref{propchechcontrhom} the maps $h^{p,q}$ form a row 
contraction of the augmented row 
$A^q (\mathfrak{U};V) \hookrightarrow \check{C}^* ( \mathfrak{U}, A^q)$. 
In addition they map continuous cochains to continuous cochains by 
construction, hence they restrict to a row contraction of the augmented sub-row 
$A_c^q (\mathfrak{U};V) \hookrightarrow \check{C}^* ( \mathfrak{U}, A_c^q)$. 
\end{proof}

\begin{definition}
  A covering $\mathfrak{U}$ of a topological space $X$ is called
  $R$-numerable, if there exists an $R$-valued partition of unity subordinate
  to $\mathfrak{U}$.
\end{definition}

\begin{corollary} \label{uqnumicindiso}
For any open covering $\mathfrak{U}=\{ U_i \mid i \in I \}$ 
of a topological space $X$ for which the coverings 
$\{ U_i^{q+1} \mid i \in I \}$ of the spaces $\mathfrak{U}[q]$ are 
$R$-numerable the homomorphism 
$i_c^*:A_c^* (\mathfrak{U};V)\rightarrow \tot\check{C}^* (\mathfrak{U},A_c^*)$ 
induces an isomorphism in cohomology.
\end{corollary}

\begin{theorem}\label{isocontuqnum}
For any open covering $\mathfrak{U}$ of a topological space $X$ for which 
each covering $\{ U_i^{q+1} \mid i \in I \}$ of $\mathfrak{U}[q]$ is 
$R$-numerable the inclusion 
$A_c^* (\mathfrak{U};V) \hookrightarrow A^* (\mathfrak{U};V)$ induces an
isomorphism in cohomology and the cohomologies 
$\check{H} (\mathfrak{U};V)$, $H_c (\mathfrak{U};V)$ and $H (\mathfrak{U};V)$ 
are isomorphic.
\end{theorem}

\begin{proof}
The inclusions $A_c (\mathfrak{U};V) \hookrightarrow A (\mathfrak{U};V)$ and 
$\tot \check{C}^* (\mathfrak{U};A_c^*) \hookrightarrow 
\tot \check{C}^* (\mathfrak{U};A^*)$ 
intertwine the augmentations $i_c^*$ and $i^*$. Thus one obtains the
following commutative diagram
\begin{equation*}
  \xymatrix{ A^*_c (\mathfrak{U};V) \ar[r]^{i_c^*} \ar[d] & 
\tot \check{C}^* ( \mathfrak{U}, A_c^*) \ar[d] & 
\check{C} (\mathfrak{U};V) \ar[l]_{j_c^*} \ar@{=}[d] \\
A^* (\mathfrak{U};V) \ar[r]^{i^*} & \tot \check{C}^* ( \mathfrak{U}, A^*) & 
\check{C} (\mathfrak{U};V) \ar[l]_{j^*} \\
}
\end{equation*}
where the horizontal arrows are induced by the augmentations and the vertical
arrows are induced by inclusion. 
The homomorphisms $j^*_c$ and $j^*$ induces isomorphisms in cohomology 
by Lemmata \ref{columnsexact} and \ref{jindiso2}. The homomorphisms $i^*$
always induces an isomorphism as observed in Corollary \ref{rowsalwexact}.
If the open coverings  $\{ U_i^{q+1} \mid i \in I \}$ of the spaces 
$\mathfrak{U}[q]$ are $R$-numerable, then the homomorphism $i^*_c$ 
also induces an isomorphism in cohomology by Corollary \ref{uqnumicindiso}. 
All in all we obtain the following commutative diagram
\begin{equation*}
  \xymatrix{ H_c (\mathfrak{U};V) \ar[r]^\cong \ar[d] & 
H ( \tot \check{C}^* ( \mathfrak{U}, A_c^*)) \ar[d] & 
\check{H} (\mathfrak{U};V) \ar[l]^\cong \ar@{=}[d] \\
H (\mathfrak{U};V) \ar[r]^\cong & 
H ( \tot \check{C}^* ( \mathfrak{U}, A^*)) & 
\check{H} (\mathfrak{U};V) \ar[l]^\cong \\
}
\end{equation*}
in which all horizontal arrows and the right vertical arrow are isomorphisms. 
This forces the homomorphism 
$H_c (\mathfrak{U};V) \rightarrow H (\mathfrak{U};V)$ induced by inclusion to
be an isomorphism as well.
\end{proof}

In this case the \v{C}ech Cohomology $\check{H} (\mathfrak{U};V)$ for the
covering $\mathfrak{U}$ of X can be either computed from the complex 
$A_c^* (\mathfrak{U};V)$ of continuous $\mathfrak{U}$-local cochains or from 
from the complex $A^* (\mathfrak{U};V)$ of $\mathfrak{U}$-local cochains.

\begin{corollary} \label{covbycoz}
For $R=\mathbb{R}$, any generalised partition of unity 
$\{ \varphi_i \mid i \in I\}$ on $X$ and 
$\mathfrak{U}:=\{ \varphi_i^{-1} (R \setminus \{ 0\}) \mid i\in I\}$ 
the inclusion 
$A_c^* (\mathfrak{U};V) \hookrightarrow A^* (\mathfrak{U};V)$ 
induces an isomorphism in cohomology and the cohomologies 
$\check{H} (\mathfrak{U};V)$, $H_c (\mathfrak{U};V)$ and $H (\mathfrak{U};V)$ 
are isomorphic.
\end{corollary}

\begin{proof}
  In view of Theorem \ref{isocontuqnum} it suffices to show that the 
coverings $\{ U_i^{q+1} \mid i \in I \}$ of the spaces $\mathfrak{U}[q]$ are 
numerable. On each space $\mathfrak{U}[q]$ we can 
(by Lemmata \ref{sumcontthensubsumcont} and 
\ref{sumvaprhicontthensumabsvalcont}) define non-negative continuous 
functions $\varphi_{q,i}$ and $\varphi_q$ via 
\begin{eqnarray*}
\varphi_{q,i} : \mathfrak{U}[q] \rightarrow \mathbb{R} & & 
\varphi_{q,i} (\vec{x})= | \varphi_{q,i} (x_0) \cdots \varphi_{q,i} (x_q) | \\
  \varphi_q : \mathfrak{U}[q] \rightarrow \mathbb{R} & & 
\varphi_q (\vec{x})= \sum_i \varphi_{q,i} (\vec{x}) \, .
\end{eqnarray*}
Since the functions $\varphi_q$ are strictly non-zero on $\mathfrak{U}[q]$, 
the set $\{ \varphi_q^{-1} \varphi_{q,i} \mid i \in I\}$ is a generalised
partition of unity with cozero sets $\{ U_i^{q+1} \mid i \in I \}$. Therefore 
the open covering $\{ U_i^{q+1} \mid i \in I \}$ of $\mathfrak{U}[q]$ is 
numerable.
\end{proof}

The colimit over all $R$-numerable coverings $\mathfrak{U}$ is called the 
\emph{Alexander-Spanier cohomology w.r.t $R$-numerable coverings}. 
Passing to the colimit over all $R$-numerable covers we observe:

\begin{corollary}
  The cohomologies $\check{H} (X;V)$, $H_{AS,c} (X;V)$ and $H_{AS} (X;V)$ 
of a topological space $X$ w.r.t. $R$-numerable coverings are isomorphic.
\end{corollary}

Calling a topological space \emph{$R$-paracompact}, if every open covering
$\mathfrak{U}$ of $X$ admits an $R$-valued partition of unity subordinate to 
$\mathfrak{U}$ we also note:

\begin{corollary}
  The cohomologies $\check{H} (X;V)$, $H_{AS,c} (X;V)$ and $H_{AS} (X;V)$ 
of an $R$-paracompact topological space $X$ are isomorphic.
\end{corollary}

These observations can especially be applied to uniform spaces 
(e.g. topological groups) $X$ with open coverings of 
the form $\mathfrak{U}_U :=\{ U[x] \mid \, x \in X\}$, where $U$ is an open 
entourage of the diagonal in $X \times X$.

\begin{proposition} \label{ubypm}
  If $d: X \times X \rightarrow \mathbb{R}$ is a continuous pseudometric on 
$X$ and $V$ a real vector space, then for each $\epsilon > 0$ and covering 
$\mathfrak{U}=\{ B_d (x,\epsilon) \mid x \in X \}$ of $X$ by open 
$\epsilon$-balls the inclusion 
$A_c^* (\mathfrak{U};V) \hookrightarrow A^* (\mathfrak{U};V)$ induces an
isomorphism in cohomology and the cohomologies 
$\check{H} (\mathfrak{U};V)$, $H_c (\mathfrak{U};V)$ and $H (\mathfrak{U};V)$ 
are isomorphic.
\end{proposition}

\begin{proof}
Let $d: X \times X \rightarrow \mathbb{R}$ be a continuous pseudometric on 
$X$ and $V$ be a real vector space. 
By \cite[Proposition B2]{S70} there exists a generalised partition of 
unity $\{ \varphi_x \mid x \in X \}$ on $X$ satisfying 
$\varphi_x^{-1} ((0,1])=  B_d (x,\epsilon)$, so Corollary \ref{covbycoz} 
applies.
\end{proof}

\begin{example}
If $\mathfrak{U}=\{ B (x,\epsilon) \mid x \in X \}$ is an open covering of a 
finite dimensional Riemannian manifold $M$ by open $\epsilon$-balls and $V$ a 
real topological vector space, 
then the cohomology $H_c (\mathfrak{U};V)$ of the complex 
$A_c^* (\mathfrak{U};V)$ is isomorphic to cohomology $H (\mathfrak{U} ;V)$ and 
to the \v{C}ech and singular cohomologies of $M$ (cf. Ex \ref{exugeodconv}). 
If $M$ is an infinite dimensional Riemannian manifold one has to require the 
local existence of geodesics.  
\end{example}

\begin{corollary}
  For any open entourage $U$ of a uniform space $X$ and real topological 
vector space $V$ 
the inclusion $A_c^* (\mathfrak{U}_U;V)\hookrightarrow A^* (\mathfrak{U}_U;V)$ 
induces an isomorphism in cohomology and the cohomologies 
$\check{H} (\mathfrak{U}_U;V)$, $H_c (\mathfrak{U}_U;V)$ and 
$H (\mathfrak{U}_U;V)$ are isomorphic.
\end{corollary}

\begin{proof}
  This follows from Corollary \ref{ubypm} and the fact that open entourages of 
uniform spaces are always of the form $U=d_U^{-1} ([0,1))$ for a continuous 
pseudometric $d_U : X \times X \rightarrow \mathbb{R}$ 
(cf.  \cite[Proposition B.2]{S70}).
\end{proof}

A particular interesting case are topological groups with open coverings of the 
form $\mathfrak{U}_U :=\{ gU \mid \, g \in G\}$, where $U$ is an open identity 
neighbourhood in $G$. Here the complexes $A_c^* (\mathfrak{U};V)$ and 
$A^* (\mathfrak{U};V)$ are sometimes called the complexes of 
\emph{continuous $U$-local cochains} and \emph{$U$-local cochains}. 
For this special case we observe:

\begin{corollary} \label{localfortopgrpsandvsv}
  For any open identity neighbourhood $U$ of a topological group $G$ and any 
real topological vector space $V$ the inclusion 
$A_c^* (\mathfrak{U}_U ;V) \hookrightarrow A^* (\mathfrak{U}_U ;V)$ induces an 
isomorphism in cohomology and the cohomologies 
$\check{H} (\mathfrak{U}_U;V)$, $H_c (\mathfrak{U}_U;V)$ and 
$H (\mathfrak{U}_U;V)$ are isomorphic.
\end{corollary}

Combining these results with those concerning singular cohomology 
(obtained in Section \ref{seccas}) we observe:

\begin{proposition} \label{propuqnumsing}
For any open covering $\mathfrak{U}$ of a topological space $X$ for which each 
set $U_{i_0 \ldots i_p}$ is $V$-acyclic and each covering 
$\{ U_i^{q+1} \mid i \in I \}$ of $\mathfrak{U}[q]$ is $R$-numerable the 
homomorphism $C (\lambda_\mathfrak{U}^* ,V) : 
A_c^* (\mathfrak{U};V) \rightarrow S^* (X,\mathfrak{U};V)$ 
induces an isomorphism $H_c (\mathfrak{U};V) \cong H_{sing} (X;V)$ in 
cohomology and the following diagram is commutative:
\begin{equation*}
  \xymatrix{
H_c (\mathfrak{U};V) \ar[r]^{H (i)}_\cong 
\ar[d]_{H (C (\lambda_\mathfrak{U}^*;V))}^\cong  &
H (\mathfrak{U};V) \ar[r]^{H (j)^{-1} H(i)}_\cong 
\ar[d]_{H (C (\lambda_\mathfrak{U}^* ;V))}^\cong & 
\check{H} (\mathfrak{U};V) \ar@{=}[d] \\
H_{sing} (\mathfrak{U};V) \ar@{=}[r] & 
H_{sing} (\mathfrak{U};V) \ar[r]^\cong_{H (j)^{-1} H(i)} &  
\check{H} (\mathfrak{U};V) 
}
\end{equation*}
In particular the \v{C}ech and the continuous $\mathfrak{U}$-local 
cohomology do not depend on the open cover $\mathfrak{U}$ subject to the 
above conditions chosen.
\end{proposition}

\begin{corollary}
 For any open covering $\mathfrak{U}$ of a topological space $X$ 
for which each set $U_{i_0 \ldots i_p}$ is $V$-acyclic and each covering 
$\{ U_i^{q+1} \mid i \in I \}$ of $\mathfrak{U}[q]$ is $R$-numerable the 
singular cohomology $H_{sing} (X;V)$ and the \v{C}ech cohomology 
$\check{H} (\mathfrak{U};V)$ for the covering $\mathfrak{U}$ 
can be computed from the complex $A^* (\mathfrak{U};V)$ of 
$\mathfrak{U}$-local cochains. 
\end{corollary}

\begin{example}
For any 'good' cover $\mathfrak{U}$ of a topological space $X$ for which
  each covering $\{ U_i^{q+1} \mid i \in I \}$ of $\mathfrak{U}[q]$ is 
$R$-numerable the morphism $ 
A_c^* (\mathfrak{U};V) \rightarrow S^* (X,\mathfrak{U};V)$ of cochain complexes
induces an isomorphism in cohomology and the cohomologies 
$\check{H} (\mathfrak{U};V)$, $H (\mathfrak{U};V)$, $H_c (\mathfrak{U};V)$ 
and $H_{sing} (X;V)$ are isomorphic.
\end{example}

\begin{lemma} \label{vacnumarecofin}
If the open coverings $\mathfrak{U}$ of a topological space $X$ for which 
the sets $U_{i_0 \ldots i_p}$ are $V$-acyclic and  each covering $\{ U_i^{q+1}
\mid i \in I \}$ of $\mathfrak{U}[q]$ is $R$-numerable are cofinal in all 
open coverings, then for each such covering $\mathfrak{U}$ the 
the cohomology $H_c (\mathfrak{U};V)$ of continuous $\mathfrak{U}$-local 
cochains coincides with the continuous Alexander-Spanier cohomology 
$H_{AS,c} (X;V)$ of $X$. In particular the directed system 
$H_c (\mathfrak{U};V)$ of abelian groups is co-Mittag-Leffler.
\end{lemma}

\begin{proof}
  In this case the (continuous) Alexander-Spanier cohomology can be 
computed as the colimit over this cofinal set of open coverings 
$\mathfrak{U}$. Proposition \ref{propuqnumsing} shows the isomorphisms 
$H_c (\mathfrak{U};V) \cong H (\mathfrak{U};V) \cong H_{sing} (X;V)$ 
for every covering $\mathfrak{U}$ in this cofinal set; this implies the
isomorphism $H_{AS,c} (X;V) \cong H_{AS} (X;V)$ of the colimit groups. It 
also shows that the directed systems $H_c (\mathfrak{U};V)$ and 
$H (\mathfrak{U};V)$ are co-Mittag-Leffler.
\end{proof}

\begin{example}
If $\mathfrak{U}$ is an open covering of a finite dimensional Riemannian 
manifold $M$ by geodetically convex sets, then the cohomology of the complex 
$A_c^* (\mathfrak{U};V)$ is isomorphic to the \v{C}ech cohomology $\check{H}
(\mathfrak{U};V)$ and to the singular cohomology $H_{sing} (M;V)$ of $M$. 
If $M$ is an infinite dimensional Riemannian manifold one has to require the 
local existence of geodesics for this argument to be applicable.
\end{example}

\begin{example}
  If $G$ is a Hilbert Lie group and $U$ a geodetically convex identity 
neighbourhood of $G$, then the cohomology of the complex 
$A_c^* (\mathfrak{U}_U;V)$ is isomorphic to \v{C}ech cohomology $\check{H}
(\mathfrak{U};V)$ and to the singular cohomology $H_{sing} (M;V)$ of $M$.
\end{example}

As observed before, one can obtain a similar result for locally contractible
topological groups without acyclicity condition on the open coverings:

\begin{theorem} \label{colimnbhfcont}
  For any locally contractible group $G$ with open neighbourhood 
filterbase $\mathcal{U}_1$ and real vector space $V$ the morphisms 
$A_c^* (\mathfrak{U}_U ;V) \hookrightarrow A^* (\mathfrak{U}_U ;V)$ and 
$C (\lambda_{\mathfrak{U}_U}^* ,V) :  A_c^* (\mathfrak{U}_U ; V) 
\rightarrow S^* (\mathfrak{U}_U ; V)$ for all $U \in \mathcal{U}_1$ induce 
isomorphisms  
\begin{equation*}
 \colim_{U \in \mathcal{U}_1} H_c (\mathfrak{U}_U ;V) \cong 
\colim_{U \in \mathcal{U}_1} H (\mathfrak{U}_U ;V) \cong H_{sing} (G;V) \, .
\end{equation*}
in cohomology. 
\end{theorem}

\begin{proof}
The first statement is a consequence of Corollary
\ref{localfortopgrpsandvsv}. The other isomorphisms follow from Theorem 
\ref{colimnbhf}.
\end{proof}

\begin{corollary}
  For Lie groups $G$ with open identity neighbourhood filter $\mathcal{U}_1$ 
and real vector space $V$ the cohomologies 
$\colim_{U \in \mathcal{U}_1} H_c (\mathfrak{U}_U ;V)$, 
$\colim_{U \in \mathcal{U}_1} H (\mathfrak{U}_U ;V)$, 
$\colim_{U \in \mathcal{U}_1} \check{H} (\mathfrak{U}_U;V)$ and 
$H_{sing} (G;V)$ coincide. 
\end{corollary}

\section{Continuous Local Cochains on $k$-Spaces}
\label{secclck}

For compactly Hausdorff generated spaces $X$ one can use abelian 
$k$-groups $V$ as coefficients and work in the subcategory $\ktop$ of
compactly Hausdorff generated spaces only. The proofs presented in Section 
\ref{secclc} carry over and one obtains analogous results for $k$-spaces and
$k$-modules. In the following we work out the details. (The properties of
$k$-spaces used here can be found in the appendix.) 

Let $X$ be a $k$-space and $V$ be a $k$-group which is a  module (in $\ktop$) 
over a  $k$-ring $R$. 
For each open covering $\mathfrak{U}:=\{ U_i \mid \, i \in I \}$ of $X$ the 
spaces $\mathfrak{U}[q]$ are open subspaces of $X^{q+1}$. Because the 
coreflector $\fktop : \tops \rightarrow \tops$ preserves open embeddings
(Lemma \ref{fktoppresopemb}) the corresponding $k$-spaces 
$\fktop \mathfrak{U}[q]$ form an open simplicial subspace of the simplicial 
$k$-space $\fktop X^{q+1}$; this allows us to define a cosimplicial abelian of 
continuous $n$-cochains with values in $V$: 
\begin{equation*}
A_{kc} (\mathfrak{U};V) := C (\fktop \mathfrak{U} ,V) 
\end{equation*} 
The cohomology of the associated cochain complex $A_{kc}^* (\mathfrak{U};V)$
is denoted by $H_{kc} (\mathfrak{U};V)$. 
Since each $k$-space $\fktop U_i^{q+1}$ is an open subspace of $X^{q+1}$, 
we can consider the presheafs $A_{kc}^q (- ;V):= C ( \fktop (-^p) , V)$ and
the sub double complex 
$\check{C}^* ( \mathfrak{U}, A_{kc}^*)$ of $\check{C}^* ( \mathfrak{U}, A^*)$
of $\check{C}^* ( \mathfrak{U}, A^*)$. 
The rows of this sub double complex $\check{C}^* ( \mathfrak{U}, A_{kc}^*)$ 
can be augmented by the complex $A_{kc}^* (\mathfrak{U};V)$ of continuous 
$\mathfrak{U}$-local cochains and the columns can be augmented by the 
\v{C}ech-complex $\check{C}^* (\mathfrak{U};V)$ for the covering 
$\mathfrak{U}$.
These augmentations induce homomorphisms 
$i_{kc}^*:A_{kc}^* (\mathfrak{U};V) \rightarrow \tot 
\check{C}^* (\mathfrak{U},A_{kc}^*)$ 
and 
$j_{kc}^*:\check{C}^* (\mathfrak{U};V) \rightarrow 
\tot \check{C}^* (\mathfrak{U}, A_{kc}^*)$ 
of cochain complexes respectively.

\begin{lemma} \label{jindisok}
  The homomorphism $j_c^*:\check{C}^* (\mathfrak{U};V) \rightarrow 
\tot \check{C}^* (\mathfrak{U}, A_c^*)$ induces an isomorphism in cohomology.
\end{lemma}

\begin{proof}
  The proof is analogous to that of Lemma \ref{columnsexact}.
\end{proof}

In the following the coefficients $V$ will always be a $k$-module  over a 
unital $k$-ring $R$. 
(Different coefficient groups are considered in Section \ref{seclcck}.)

\begin{proposition} \label{propchechcontrhomcontk}
For any $R$-valued partition of unity $\{ \varphi_{q,i} \mid i \in I\}$ 
subordinate to the covering $\{ \fktop U_i^{q+1} \mid i \in I \}$ of 
$\fktop \mathfrak{U}[q]$ the homomorphisms 
\begin{equation} \label{eqchechcontrhomcontk}
  h^{p,q} : \check{C}^p ( \mathfrak{U}, A^q) \rightarrow 
\check{C}^{p-1} ( \mathfrak{U}, A^q), \quad
h^{p,q} ( f )_{i_0 \ldots i_{p-1}} = 
\sum_{i} \varphi_{q,i} \cdot f_{i i_0 \ldots i_{p-1}}
\end{equation}
form a row contraction of the augmented row 
$A^q (\mathfrak{U};V) \hookrightarrow \check{C}^* ( \mathfrak{U}, A^q)$ which
restricts to a row contraction of the augmented sub-row 
$A_{kc}^q (\mathfrak{U};V)\hookrightarrow \check{C}^* ( \mathfrak{U},A_{kc}^q)$.
\end{proposition}

\begin{proof}
  The proof is analogous to that of Proposition \ref{propchechcontrhomcont}.
\end{proof}

\begin{corollary} \label{uqnumicindisok}
For any open covering $\mathfrak{U}=\{ U_i \mid i \in I \}$ 
of a $k$-space $X$ for which the coverings 
$\{ \fktop U_i^{q+1} \mid i \in I \}$ of the $k$-spaces $\mathfrak{U}[q]$ are 
$R$-numerable the homomorphism 
$i_{kc}^*: A_{kc}^* (\mathfrak{U};V)\rightarrow 
\tot\check{C}^* (\mathfrak{U},A_{kc}^*)$ induces an isomorphism in cohomology.
\end{corollary}

\begin{theorem}\label{isocontuqnumk}
For any open covering $\mathfrak{U}$ of a $k$-space $X$ for which 
each covering $\{ \fktop U_i^{q+1} \mid i \in I \}$ of 
$\fktop \mathfrak{U}[q]$ is $R$-numerable the inclusion 
$A_{kc}^* (\mathfrak{U};V) \hookrightarrow A^* (\mathfrak{U};V)$ induces an
isomorphism in cohomology and the cohomologies 
$\check{H} (\mathfrak{U};V)$, $H_{kc} (\mathfrak{U};V)$ and 
$H (\mathfrak{U};V)$ are isomorphic.
\end{theorem}

\begin{proof}
  The proof is analogous to that of Theorem \ref{isocontuqnum}.
\end{proof}

In this case the \v{C}ech Cohomology $\check{H} (\mathfrak{U};V)$ for the
covering $\mathfrak{U}$ of X can be either computed from the complex 
$A_{kc}^* (\mathfrak{U};V)$ of continuous $\mathfrak{U}$-local cochains or from 
from the complex $A^* (\mathfrak{U};V)$ of $\mathfrak{U}$-local cochains.

\begin{corollary} \label{covbycozk}
If the ring $R$ is a complete $k$-field, 
$\{ \varphi_i \mid i \in I\}$ a generalised $R$-valued partition of unity on 
$X$ and $\mathfrak{U}:=\{ \varphi_i^{-1} (R \setminus \{ 0\}) \mid i\in I\}$ 
then the inclusion 
$A_c^* (\mathfrak{U};V) \hookrightarrow A^* (\mathfrak{U};V)$ 
induces an isomorphism in cohomology and the cohomologies 
$\check{H} (\mathfrak{U};V)$, $H_{kc} (\mathfrak{U};V)$ and 
$H (\mathfrak{U};V)$ are isomorphic.
\end{corollary}

If the products $X^{q+1}$ in $\tops$ are already compactly Hausdorff 
generated, then the open subspaces $\mathfrak{U}[q]$ also are $k$-spaces. 
This in particular happens if $X$ is a metric space, locally compact or a 
Hausdorff $k_\omega$-space (cf. Lemma \ref{appfinprodofko}). 
In this case the diagonal neighbourhoods of the form $\mathfrak{U}[q]$ are 
cofinal in all diagonal neighbourhoods and the colimit 
$\colim H_{kc} (\mathfrak{U};V)$ over all $R$-numerable coverings 
$\mathfrak{U}$ is called the \emph{Alexander-Spanier cohomology w.r.t 
$R$-numerable coverings}. 

\begin{corollary}
For metric, locally compact or Hausdorff $k_\omega$-spaces $X$ the
cohomologies $\check{H} (X;V)$, $H_{AS,kc} (X;V)$ and $H_{AS} (X;V)$ 
w.r.t. $R$-numerable coverings are isomorphic.
\end{corollary}

\begin{example}
Real and complex Kac-Moody groups are Hausdorff $k_\omega$-spaces 
(cf.  \cite{GGH06}). Thus for real or complex Kac-Moody groups $G$ the 
cohomologies $\check{H} (G;V)$, $H_{AS,kc} (G;V)$ and $H_{AS} (G;V)$ w.r.t. 
$R$-numerable coverings are isomorphic.
\end{example}

\begin{corollary}
  The cohomologies $\check{H} (X;V)$, $H_{AS,kc} (X;V)$ and $H_{AS} (X;V)$ 
of a metric, paracompact and locally compact or paracompact Hausdorff 
$k_\omega $-space $X$ are isomorphic.
\end{corollary}

\begin{example}
For metrisable manifolds $M$ and real $k$-modules $V$ the cohomologies 
$\check{H} (M;V)$, $H_{AS,kc} (M;V)$ and $H_{AS} (M;V)$ are isomorphic. 
\end{example}

\begin{proposition} \label{ubypmk}
  If $d: X \times X \rightarrow \mathbb{R}$ is a continuous pseudometric on 
$X$ and $V$ a real $k$-module then for each $\epsilon > 0$ and covering 
$\mathfrak{U}=\{ B_d (x,\epsilon) \mid x \in X \}$ of $X$ by open 
$\epsilon$-balls the inclusion 
$A_{kc}^* (\mathfrak{U};V) \hookrightarrow A^* (\mathfrak{U};V)$ induces an
isomorphism in cohomology and the cohomologies 
$\check{H} (\mathfrak{U};V)$, $H_{kc} (\mathfrak{U};V)$ and 
$H (\mathfrak{U};V)$ are isomorphic.
\end{proposition}

\begin{proof}
  The proof is analogous to that of Proposition \ref{ubypm}.
\end{proof}

\begin{example}
If $\mathfrak{U}=\{ B (x,\epsilon) \mid x \in X \}$ is an open covering of a 
finite dimensional Riemannian manifold $M$ by open $\epsilon$-balls and $V$ a 
real $k$-module, 
then the cohomology $H_{kc} (\mathfrak{U};V)$ of the complex 
$A_{kc}^* (\mathfrak{U};V)$ is isomorphic to cohomology 
$H (\mathfrak{U} ;V)$ and to the \v{C}ech and singular cohomologies of $M$ 
(cf. Ex. \ref{exugeodconv}). If $M$ is an infinite dimensional Riemannian 
manifold one has to require the local existence of geodesics.  
\end{example}

\begin{corollary}
For any open entourage $U$ of a uniform $k$-space $X$ and real $k$-module $V$ 
the inclusion 
$A_{kc}^* (\mathfrak{U}_U;V)\hookrightarrow A^* (\mathfrak{U}_U;V)$ 
induces an isomorphism in cohomology and the cohomologies 
$\check{H} (\mathfrak{U}_U;V)$, $H_{kc} (\mathfrak{U}_U;V)$ and 
$H (\mathfrak{U}_U;V)$ are isomorphic.
\end{corollary}

A particular interesting case are compactly Hausdorff generated topological 
groups $G$. 
For such groups and $k$-modules $V$ over the $k$-ring $\mathbb{R}$ we observe:

\begin{corollary} \label{localfortopgrpsandvsvk}
For any open $1$-neighbourhood $U$ of a compactly Hausdorff generated 
topological group $G$ and real $k$-module $V$ the inclusion 
$A_{kc}^* (\mathfrak{U}_U ;V) \hookrightarrow A^* (\mathfrak{U}_U ;V)$
induces an isomorphism in cohomology and the cohomologies 
$\check{H} (\mathfrak{U}_U;V)$, $H_c (\mathfrak{U}_U;V)$ and 
$H (\mathfrak{U}_U;V)$ are isomorphic.
\end{corollary}

This especially applies to all metric groups, all locally compact groups and
all topological groups which are Hausdorff $k_\omega$-spaces:

\begin{example}
  For any Hilbert Lie group, real or complex Kac-Moody group $G$ and real
  $k$-module $V$ the inclusion 
$A_{kc}^* (\mathfrak{U}_U ;V) \hookrightarrow A^* (\mathfrak{U}_U ;V)$ 
induces an isomorphism in cohomology and the cohomologies 
$\check{H} (\mathfrak{U}_U;V)$, $H_c (\mathfrak{U}_U;V)$ and 
$H (\mathfrak{U}_U;V)$ are isomorphic.
\end{example}

\begin{proposition} \label{propuqnumsingk}
For any open covering $\mathfrak{U}$ of a $k$-space $X$ for which each 
set $U_{i_0 \ldots i_p}$ is $V$-acyclic and each covering 
$\{ \fktop U_i^{q+1} \mid i \in I \}$ of $\fktop \mathfrak{U}[q]$ is 
$R$-numerable the homomorphism $C (\lambda_\mathfrak{U}^* ,V) : 
A_{kc}^* (\mathfrak{U};V) \rightarrow S^* (X,\mathfrak{U};V)$ 
induces an isomorphism $H_{kc} (\mathfrak{U};V) \cong H_{sing} (X;V)$ in 
cohomology and the following diagram is commutative:
\begin{equation*}
  \xymatrix{
H_{kc} (\mathfrak{U};V) \ar[r]^{H (i)}_\cong 
\ar[d]_{H (C (\lambda_\mathfrak{U}^*;V))}^\cong  &
H (\mathfrak{U};V) \ar[r]^{H (j)^{-1} H(i)}_\cong 
\ar[d]_{H (C (\lambda_\mathfrak{U}^* ;V))}^\cong & 
\check{H} (\mathfrak{U};V) \ar@{=}[d] \\
H_{sing} (\mathfrak{U};V) \ar@{=}[r] & 
H_{sing} (\mathfrak{U};V) \ar[r]^\cong_{H (j)^{-1} H(i)} & 
\check{H} (\mathfrak{U};V) 
}
\end{equation*}
In particular the \v{C}ech and the continuous $\mathfrak{U}$-local 
cohomology do not depend on the open cover $\mathfrak{U}$ subject to the 
above conditions chosen.
\end{proposition}

\begin{corollary}
 For any open covering $\mathfrak{U}$ of a $k$-space $X$ 
for which each set $U_{i_0 \ldots i_p}$ is $V$-acyclic and each covering 
$\{ \fktop U_i^{q+1} \mid i \in I \}$ of $\fktop \mathfrak{U}[q]$ is 
$R$-numerable the singular cohomology $H_{sing} (X;V)$ and the \v{C}ech 
cohomology $\check{H} (\mathfrak{U};V)$ for the covering $\mathfrak{U}$ 
can be computed from the complex $A_{kc}^* (\mathfrak{U};V)$ of 
continuous $\mathfrak{U}$-local cochains. 
\end{corollary}

\begin{example}
For any 'good' cover $\mathfrak{U}$ of a $k$-space $X$ for which each covering 
$\{ \fktop U_i^{q+1} \mid i \in I \}$ of 
$\fktop \mathfrak{U}[q]$ is $R$-numerable the morphism 
$ 
A_{kc}^* (\mathfrak{U};V) \rightarrow S^* (X,\mathfrak{U};V)$ 
of cochain complexes induces an isomorphism in cohomology and the cohomologies 
$\check{H} (\mathfrak{U};V)$, $H (\mathfrak{U};V)$, $H_{kc} (\mathfrak{U};V)$ 
and $H_{sing} (X;V)$ are isomorphic.
\end{example}

\begin{lemma} \label{vacnumarecofink}
If $X$ is a $k$-space and the diagonal neighbourhoods $\mathfrak{U}[q]$ 
for open coverings $\mathfrak{U}$ for which the sets 
$U_{i_0 \ldots i_p}$ are $V$-acyclic and each covering 
$\{ \fktop U_i^{q+1} \mid i\in I\}$ of $\fktop \mathfrak{U}[q]$ 
is $R$-numerable are cofinal in all diagonal neighbourhoods, 
then for each such covering $\mathfrak{U}$ the cohomology 
$H_{kc} (\mathfrak{U};V)$ of continuous $\mathfrak{U}$-local 
cochains coincides with the continuous Alexander-Spanier cohomology 
$H_{AS,kc} (X;V)$ of $X$. In particular the directed system 
$H_{kc} (\mathfrak{U};V)$ of abelian groups is co-Mittag-Leffler.
\end{lemma}

\begin{example}
If $\mathfrak{U}$ is an open covering of a finite dimensional Riemannian 
manifold $M$ by geodetically convex sets, then the cohomology of the complex 
$A_{kc}^* (\mathfrak{U};V)$ is isomorphic to the \v{C}ech cohomology $\check{H}
(\mathfrak{U};V)$ and to the singular cohomology $H_{sing} (M;V)$ of $M$. 
If $M$ is an infinite dimensional Riemannian manifold one has to require the 
local existence of geodesics for this argument to be applicable.
\end{example}

\begin{example}
  If $G$ is a Hilbert Lie group and $U$ a geodetically convex identity 
neighbourhood of $G$, then the cohomology of the complex 
$A_{kc}^* (\mathfrak{U}_U;V)$ is isomorphic to \v{C}ech cohomology $\check{H}
(\mathfrak{U};V)$ and to the singular cohomology $H_{sing} (M;V)$ of $M$.
\end{example}

One obtains a similar result for locally contractible compactly Hausdorff
generated topological groups without acyclicity condition on the open 
coverings:

\begin{theorem} \label{colimnbhfcontk}
For any locally contractible compactly Hausdorff generated group $G$ with 
open neighbourhood filterbase $\mathcal{U}_1$ for which all finite products 
$G^{p+1}$ are $k$-spaces and any real $k$-module $V$ the morphisms 
$A_{kc}^* (\mathfrak{U}_U ;V) \hookrightarrow A^* (\mathfrak{U}_U ;V)$ and 
$C (\lambda_{\mathfrak{U}_U}^* ,V) :  A_{kc}^* (\mathfrak{U}_U ; V) 
\rightarrow S^* (\mathfrak{U}_U ; V)$ for all $U \in \mathcal{U}_1$ induce 
isomorphisms  
\begin{equation*}
 \colim_{U \in \mathcal{U}_1} H_{kc} (\mathfrak{U}_U ;V) \cong 
\colim_{U \in \mathcal{U}_1} H (\mathfrak{U}_U ;V) \cong H_{sing} (G;V) \, .
\end{equation*}
in cohomology. 
\end{theorem}

\begin{corollary}
For metrisable Lie groups $G$ with open identity neighbourhood 
filter $\mathcal{U}_1$ and real $k$-modules $V$ the cohomologies 
$\colim_{U \in \mathcal{U}_1} H_{kc} (\mathfrak{U}_U ;V)$, 
$\colim_{U \in \mathcal{U}_1} H (\mathfrak{U}_U ;V)$, 
$\colim_{U \in \mathcal{U}_1} \check{H} (\mathfrak{U}_U;V)$ and 
$H_{sing} (G;V)$ coincide. 
\end{corollary}

\section{Smooth Local Cochains}
\label{secslc}

We use the differential calculus over general base fields and rings 
presented in \cite{BGN04} and assume that the ring $R$ to be a smooth manifold 
with smooth addition and multiplication.
For an open covering $\mathfrak{U}$ of a (possibly infinite dimensional) 
differential manifold $M$ and abelian Lie groups $V$ which are smooth
$R$-modules one can consider the complex 
$A_s^* (\mathfrak{U};V) = C^\infty (\mathfrak{U} [*] ,V)$ of smooth 
$\mathfrak{U}$-local cochains. 
Replacing the pre-sheaf $A^q (-;V)$ of $q$-cochains 
by the pre-sheaf $A_s^q (-;V)=C^\infty ( -^q ;V)$ of smooth $q$-cochains we 
obtain a sub double complex 
$\check{C}^* ( \mathfrak{U}, A_s^*)$ of $\check{C}^* ( \mathfrak{U}, A_c^*)$ 
whose groups are given by
\begin{equation*}
\check{C}^p ( \mathfrak{U}, A_s^q) :=
\left\{ f \in \check{C}^p ( \mathfrak{U}, A^q) \mid \forall i_0 , \ldots i_p
  \in I: f_{i_0 \ldots i_p} \in C^\infty ( U_{i_0 \ldots i_p}^{q+1} ;V ) \right\}
\, .
\end{equation*}
The rows of this sub double complex $\check{C}^* ( \mathfrak{U}, A_s^*)$ 
can be augmented by the complex $A_s^* (\mathfrak{U};V)$ of smooth  
$\mathfrak{U}$-local cochains and the columns can be augmented by the 
\v{C}ech-complex $\check{C}^* (\mathfrak{U};V)$ for the covering 
$\mathfrak{U}$.
These augmentations induce homomorphisms 
$i_s^*:A_s^* (\mathfrak{U};V) \rightarrow 
\tot \check{C}^* (\mathfrak{U},A_c^*)$ 
and 
$j_s^*:\check{C}^* (\mathfrak{U};V) \rightarrow 
\tot \check{C}^* (\mathfrak{U}, A_s^*)$ 
of cochain complexes respectively.

\begin{lemma} \label{jindiso3}
  The homomorphism $j_s^*:\check{C}^* (\mathfrak{U};V) \rightarrow 
\tot \check{C}^* (\mathfrak{U}, A_s^*)$ induces an isomorphism in cohomology.
\end{lemma}

\begin{proof}
  The proof is analogous to that of Lemma \ref{columnsexact}.
\end{proof}

Replacing continuity by smoothness, $R$-paracompactness by smooth 
$R$-para\-com\-pactness in the discussion in Section \ref{secclc} yields:

\begin{proposition} \label{propchechcontrhomsmooth}
For any smooth $R$-valued partition of unity $\{ \varphi_{q,i} \mid i \in I\}$ 
subordinate to the covering $\{U_i^{q+1} \mid i \in I \}$ of $\mathfrak{U}[q]$ 
the homomorphisms 
\begin{equation} \label{eqchechcontrhomsmooth}
  h^{p,q} : \check{C}^p ( \mathfrak{U}, A^q) \rightarrow 
\check{C}^{p-1} ( \mathfrak{U}, A^q), \quad
h^{p,q} ( f )_{i_0 \ldots i_{p-1}} = 
\sum_{i} \varphi_{q,i} \cdot f_{i i_0 \ldots i_{p-1}}
\end{equation}
form a row contraction of the augmented row 
$A^q (\mathfrak{U};V) \hookrightarrow \check{C}^* ( \mathfrak{U}, A^q)$ which
restricts to a row contraction of the augmented sub-row 
$A_s^q (\mathfrak{U};V) \hookrightarrow \check{C}^* ( \mathfrak{U}, A_s^q)$.
\end{proposition}

\begin{proof}
  The proof is analogous to that of Proposition \ref{propchechcontrhomcont}.
\end{proof}

\begin{corollary} \label{uqnumicindisosmooth}
For any open covering $\mathfrak{U}=\{ U_i \mid i \in I \}$ 
of a manifold $M$ for which the coverings $\{ U_i^{q+1} \mid i \in I \}$ of 
the manifolds $\mathfrak{U}[q]$ are smoothly $R$-numerable the homomorphism 
$i_s^*:A_s^* (\mathfrak{U};V)\rightarrow \tot\check{C}^* (\mathfrak{U},A_s^*)$ 
induces an isomorphism in cohomology.
\end{corollary}

\begin{theorem} \label{isosmoothuqnum}
For any open covering $\mathfrak{U}$ of a manifold $M$ for which each covering 
$\{ U_i^{q+1} \mid i \in I \}$ of $\mathfrak{U}[q]$ is smoothly $R$-numerable 
the inclusion $A_s^* (\mathfrak{U};V) \hookrightarrow A^* (\mathfrak{U};V)$ 
induces an isomorphism in cohomology and the cohomologies 
$\check{H} (\mathfrak{U};V)$, $H_s (\mathfrak{U};V)$, $H_c (\mathfrak{U};V)$ 
and $H (\mathfrak{U};V)$ are isomorphic.
\end{theorem}

\begin{proof}
  The proof is analogous to that of Theorem \ref{isocontuqnum}.
\end{proof}

In this case the \v{C}ech Cohomology $\check{H} (\mathfrak{U};V)$ for the
covering $\mathfrak{U}$ of X can be either computed from the complex 
$A_s^* (\mathfrak{U};V)$ of smooth $\mathfrak{U}$-local cochains, the complex 
$A_c^* (\mathfrak{U};V)$ of continuous $\mathfrak{U}$-local cochains or from 
from the complex $A^* (\mathfrak{U};V)$ of $\mathfrak{U}$-local cochains.

\begin{corollary} \label{covbycozsmooth}
For any smooth $R$-valued partition of unity $\{ \varphi_i \mid i \in I\}$ on 
$M$ and $\mathfrak{U}:=\{ \varphi_i^{-1} (R \setminus \{ 0 \}) \mid i \in I \}$ 
the inclusion 
$A_s^* (\mathfrak{U};V) \hookrightarrow A^* (\mathfrak{U};V)$ 
induces an isomorphism in cohomology and the cohomologies 
$\check{H} (\mathfrak{U};V)$, $H_s (\mathfrak{U};V)$, $H_c (\mathfrak{U};V)$ 
and $H (\mathfrak{U};V)$ are isomorphic.
\end{corollary}

\begin{proof}
  The proof is analogous to the proof of Corollary \ref{covbycoz} where the 
requirement of local finiteness guarantees the smoothness of the functions 
$\varphi_q$.
\end{proof}

For the ring $R=\mathbb{R}$ of reals we obtain a more general version:

\begin{corollary} \label{covbycozsmoothg}
For $R=\mathbb{R}$, any smooth generalised partition of unity 
$\{ \varphi_i \mid i \in I\}$ on $M$ and 
$\mathfrak{U}:=\{ \varphi_i^{-1} (R \setminus \{ 0 \}) \mid i \in I \}$ 
the inclusion 
$A_s^* (\mathfrak{U};V) \hookrightarrow A^* (\mathfrak{U};V)$ 
induces an isomorphism in cohomology and the cohomologies 
$\check{H} (\mathfrak{U};V)$, $H_s (\mathfrak{U};V)$, $H_c (\mathfrak{U};V)$ 
and $H (\mathfrak{U};V)$ are isomorphic.
\end{corollary}

\begin{proof}
 In view of Theorem \ref{isosmoothuqnum} it suffices to show that the 
coverings $\{ U_i^{q+1} \mid i \in I \}$ of the manifolds $\mathfrak{U}[q]$ are 
smoothly numerable. The smooth functions $\varphi_{i_1,\ldots,i_q}$ given by 
$\varphi_{i_0,\ldots,i_q} (\vec{m}):=\varphi_{i_1} (m_0) \cdots \varphi_{i_q}
(m_q)$, $i_0,\ldots,i_q \in I$ form a generalised partition of unity of
$M^{q+1}$. By Lemma \ref{genpartthenvarphiin} there exist there exist 
non-negative smooth real functions 
$\{ \varphi_{i_0,\ldots , i_q ,n} \mid i_0,\ldots i_q \in I,\, n \in \mathbb{N}\}$ 
such that for all $n \in \mathbb{N}$  
\begin{enumerate}
\item the collection 
$\{ \supp \varphi_{\vec{i} ,n} \mid \vec{i} \in I^{q+1} \}$ refines 
$\{ U_{i_0} \times \cdots \times U_{i_q} \mid i_0, \ldots , i_q \in I \}$,
\item the collection 
$\{ \supp \varphi_{\vec{i} , n}  \mid \vec{i} \in I^{q+1} \}$ of
supports is locally finite, 
\end{enumerate}
and such that that for for fixed $\vec{i} \in I^{q+1}$ the supports of 
$\varphi_{\vec{i},n}$, $n \in \mathbb{N}$ exhaust the open set $U_{i_0} \times
\cdots \times U_{i_q}$. An application of Proposition 
\ref{uisnumiffvarphiin} to the set $\{ \varphi_{i,\ldots,i} \mid i \in I\}$ of
smooth functions shows that the covering $\{ U_i^{q+1} \mid i \in I \}$ of
$\mathfrak{U}[q]$ is numerable.
\end{proof}

The colimit $H_{AS,s} (M;V): = \colim H (\mathfrak{U};V)$ over all smoothly 
$R$-numerable coverings $\mathfrak{U}$ is called 
the \emph{Alexander-Spanier cohomology w.r.t smoothly $R$-numerable coverings}. 
Passing to the colimit over all smoothly $R$-numerable covers we observe:

\begin{corollary}
  The cohomologies $\check{H} (M;V)$, $H_{AS,s} (M;V)$ and 
$H_{AS} (M;V)$ of a manifold $M$ w.r.t. smoothly $R$-numerable coverings 
are isomorphic.
\end{corollary}

Calling a manifold smoothly \emph{$R$-paracompact}, if every open 
covering $\mathfrak{U}$ of $X$ admits a smooth $R$-valued partition of unity 
subordinate to $\mathfrak{U}$ we also note:

\begin{corollary}
  The cohomologies $\check{H} (M;V)$, $H_{AS,s} (M;V)$ and 
$H_{AS} (M;V)$ of a smoothly $R$-paracompact manifold $M$ are isomorphic.
\end{corollary}

\begin{example}
If $\mathfrak{U}$ is an open covering of a Riemannian manifold $M$ and $V$ a 
real topological vector space, then the cohomology $H_s (\mathfrak{U};V)$ of 
the complex $A_s^* (\mathfrak{U};V)$ is isomorphic to cohomology 
$H (\mathfrak{U} ;V)$. If the open sets $U \in \mathfrak{U}$ are geodsically
convex then these cohomologies are also isomorphic to the \v{C}ech and 
singular cohomologies of $M$ (cf. Ex \ref{exugeodconv}). 
(If $M$ is an infinite dimensional Riemannian manifold one has to require the 
local existence of geodesics.)
\end{example}

A particular interesting case are Lie groups with open coverings of the 
form $\mathfrak{U}_U :=\{ gU \mid \, g \in G\}$, where $U$ is an open identity 
neighbourhood in $G$. Here the complex $A_s^* (\mathfrak{U};V)$ is sometimes 
called the complex of \emph{smooth $U$-local cochains}. For this special case 
we observe:

\begin{corollary} \label{localfortopgrpsandvsvsmooth}
  For any open identity neighbourhood $U$ of a smoothly paracompact Lie group 
$G$ and any real TVS $V$ the inclusion 
$A_s^* (\mathfrak{U}_U ;V) \hookrightarrow A^* (\mathfrak{U}_U ;V)$ induces an 
isomorphism in cohomology and the cohomologies 
$\check{H} (\mathfrak{U}_U;V)$, $H_s (\mathfrak{U}_U;V)$, 
$H_c (\mathfrak{U}_U;V)$ and $H (\mathfrak{U}_U;V)$ are isomorphic.
\end{corollary}

Combining these results with those concerning singular cohomology we observe:

\begin{proposition}  \label{propuqsmoothnumsing}
For any open covering $\mathfrak{U}$ of a smooth manifold $M$  for which each 
set $U_{i_0 \ldots i_p}$ is $V$-acyclic and each covering 
$\{ U_i^{q+1} \mid i \in I \}$ of $\mathfrak{U}[q]$ is smoothly $R$-numerable 
the homomorphism $ 
A_s^* (\mathfrak{U};V) \rightarrow S^* (\mathfrak{U};V)$ 
induces an isomorphism $H_s (\mathfrak{U};V) \cong H_{sing} (X;V)$ in 
cohomology and the following diagram is commutative:
\begin{equation*}
  \xymatrix{
H_s (\mathfrak{U};V) \ar[r]_\cong 
\ar[d]_{H (C (\lambda_\mathfrak{U}^*;V))}^\cong  &
H_c (\mathfrak{U};V) \ar[r]_\cong 
\ar[d]_{H (C (\lambda_\mathfrak{U}^*;V))}^\cong  &
H (\mathfrak{U};V) \ar[r]^{H (j)^{-1} H(i)}_\cong 
\ar[d]_{H (C (\lambda_\mathfrak{U}^* ;V))}^\cong & 
\check{H} (\mathfrak{U};V) \ar@{=}[d] \\
H_{sing} (\mathfrak{U};V) \ar@{=}[r] &
H_{sing} (\mathfrak{U};V) \ar@{=}[r] & 
H_{sing} (\mathfrak{U};V) \ar[r]^\cong_{H (j)^{-1} H(i)} & 
\check{H} (\mathfrak{U};V) 
}
\end{equation*}
In particular the \v{C}ech and the smooth $\mathfrak{U}$-local 
cohomology do not depend on the open cover $\mathfrak{U}$ subject to the 
acyclicity condition chosen.
\end{proposition}

\begin{proof}
  The proof is analogous to the proof of Proposition \ref{propuqnumsing}.
\end{proof}

\begin{corollary}
 For any open covering $\mathfrak{U}$ of a smooth manifold $M$ for which each
 set $U_{i_0 \ldots i_p}$ is $V$-acyclic and each covering 
$\{ U_i^{q+1} \mid i \in I \}$ of $\mathfrak{U}[q]$ is smoothly $R$-numerable 
the singular cohomology $H_{sing} (M;V)$ and the \v{C}ech cohomology
$\check{H} (\mathfrak{U};V)$ 
can both be computed from the complex $A_s^* (\mathfrak{U};V)$ of 
smooth $\mathfrak{U}$-local cochains. 
\end{corollary}

\begin{example}
  For any 'good' covering $\mathfrak{U}=\{ U_i \mid i \in I \}$ of a manifold 
$M$ for which each covering $\{ U_i^{q+1} \mid i \in I \}$ of 
$\mathfrak{U}[q]$ is smoothly $R$-numerable the morphism 
$ C (\lambda_\mathfrak{U}^* ,V) : 
A_s^* (\mathfrak{U};V) \rightarrow S^* (\mathfrak{U};V)$ of cochain complexes
induces an isomorphism in cohomology and the cohomologies 
$\check{H} (\mathfrak{U};V)$, $H_s (\mathfrak{U};V)$, $H_c (\mathfrak{U};V)$, 
$H (\mathfrak{U};V)$ and $H_{sing} (M;V)$ are isomorphic.
\end{example}

\begin{lemma} \label{vacnumarecofinsmooth}
If the open coverings $\mathfrak{U}$ of a manifold $M$ for which the sets 
$U_{i_0 \ldots i_p}$ are $V$-acyclic and  each covering 
$\{ U_i^{q+1} \mid i \in I \}$ of $\mathfrak{U}[q]$ is smoothly $R$-numerable 
are cofinal in all open coverings, then for each such covering $\mathfrak{U}$
the the cohomology $H_s (\mathfrak{U};V)$ of smooth $\mathfrak{U}$-local 
cochains coincides with the smooth Alexander-Spanier cohomology 
$H_{AS,s} (M;V)$ of $M$. In particular the directed system 
$H_s (\mathfrak{U};V)$ of abelian groups is co-Mittag-Leffler.
\end{lemma}

\begin{proof}
  In this case the (smooth) Alexander-Spanier cohomology can be 
computed as the colimit over this cofinal set of open coverings 
$\mathfrak{U}$. Proposition \ref{propuqsmoothnumsing} shows the isomorphisms 
$H_c (\mathfrak{U};V) \cong H (\mathfrak{U};V) \cong H_{sing} (M;V)$ 
for every covering $\mathfrak{U}$ in this cofinal set; this implies the
isomorphism $H_{AS,c} (M;V) \cong H_{AS} (M;V)$ of the colimit groups. It 
also shows that the directed systems $H_c (\mathfrak{U};V)$ and 
$H (\mathfrak{U};V)$ are co-Mittag-Leffler.
\end{proof}

\begin{example}
If $\mathfrak{U}$ is an open covering of a finite dimensional Riemannian 
manifold $M$ by geodetically convex sets, then the cohomology of the complex 
$A_s^* (\mathfrak{U};V)$ is isomorphic to the \v{C}ech cohomology $\check{H}
(\mathfrak{U};V)$ and to the singular cohomology $H_{sing} (M;V)$ of $M$. 
If $M$ is an infinite dimensional Riemannian manifold one has to require the 
local existence of geodesics and smooth partitions of unity for this argument 
to be applicable.
\end{example}

\begin{example}
  If $G$ is a Hilbert Lie group and $U$ a geodetically convex identity 
neighbourhood of $G$, then the cohomology of the complex 
$A_s^* (\mathfrak{U}_U;V)$ is isomorphic to \v{C}ech cohomology $\check{H}
(\mathfrak{U};V)$ and to the singular cohomology $H_{sing} (M;V)$ of $M$.
\end{example}

For Lie groups one can also consider the complex 
$\colim_{U \in \mathcal{U}_1} A_s^* (\mathfrak{U}_U ;V)$, where $U$ ranges
over all open identity neighbourhoods of $G$. In order to obtain a result
similar to that for locally contractible groups (Theorem \ref{colimnbhfcont}) we
require the Lie groups and their finite products to be smoothly 
$R$-paracompact. Although this does not guarantee the row-exactness of the 
double complex $\check{C}^* (\mathfrak{U}_U, A_c^*)$, it allows us to construct 
approximative row contractions 
$h^{*,q} : \check{C}^* (\mathfrak{U}_U, A_s^q) \rightarrow  
\check{C}^{*-1} (\mathfrak{U}_U, A_s^q)$. These row contractions will serve the 
same purpose after shrinking the open covering $\mathfrak{U}_U$ to an open 
covering $\mathfrak{U}_V$.

\begin{lemma} \label{foropenuexistsvandphi}
For any open identity neighbourhood $U$ of a Lie group $G$ for which all
finite products are smoothly $R$-paracompact, there exists an open identity 
neighbourhood $V \subseteq U$ and smooth $R$-valued functions 
$\{\varphi_{q,g} \mid G \in G\}$ with locally finite supports in 
$(gU) \times \cdots \times (gU)$ respectively such that the restriction of 
each function $\varphi_q = \sum_{g \in G} \varphi_{q,g}$ to $\mathfrak{V}[q]$
is the constant function $1$.
\end{lemma}

\begin{proof}
Let $V$ and $W$ be open identity neighbourhoods satisfying 
$V^{-1} V \subseteq \overline{W} \subseteq U$. 
We assert that the closure of $\mathfrak{V}[q]$ is contained in 
$\mathfrak{U}[q]$. If $\vec{x}_j$ is a net in $\mathfrak{V}[q]$ converging to 
a point $\vec{x}$ in $G^{q+1}$, then the net $x_{j,0}^{-1} \vec{x}_j$
converges to $x_0^{-1} \vec{x}$. Every point $\vec{x}_j$ is contained in some 
open set $(gV)\times \cdots \times (gV)$, as a consequence the point 
$x_{j,0}^{-1} x_{j,k}$ is contained in 
$(gV)^{-1} gV = V^{-1} g^{-1} g V = V^{-1} V \subset \overline{W}$ and 
$\vec{x}$ is contained in 
$(x_0 \overline{W} ) \times \cdots \times (x_0\overline{W})$; the latter set
is contained in $\mathfrak{U}[q]$. Thus the complement 
$A=G^{q+1} \setminus \overline{\mathfrak{V}[q]}$ and the open sets 
$(gU) \times \cdots \times (gU)$ form an open covering of $G^{q+1}$. If the 
Lie group $G$ and its finite products are smoothly $R$-paracompact, then there
exists a smooth $R$-valued partition of unity 
$\{ \psi\} \cup \{\varphi_{q,g} \mid g \in G\}$ subordinate to this open 
covering. Since the function $\psi$ has support in the complement of
$\mathfrak{V}[q]$ the sum $\varphi_q = \sum_{g \in G} \varphi_{q,g}$ is $1$ on 
$\mathfrak{V}[q]$.
\end{proof}

For any open covering $\mathfrak{U}:=\{ U_i \mid i \in I\}$ of a smooth
manifold $M$, set of smooth $R$-valued functions 
$\{\varphi_{q,i} \mid i\in I\}$ with locally finite supports contained in 
$U_i^{q+1}$ respectively and smooth cochain 
$f \in \check{C}^p (\mathfrak{U};A_s^q)$ the products 
$\varphi_{q,i} f_{i  i_0 \ldots i_{p-1}}$ have supports in the open sets 
$U_{i i_0 \ldots i_p}^{q+1}$ respectively. Therefore they can be smoothly 
extended to $U_{i_0 \ldots i_{p-1} }^{q+1}$ by defining it to be zero outside 
$U_{i i_0 \ldots i_{p-1}}^{q+1}$. Understanding each function 
$\varphi_{q,i} f_{i g_0 \ldots i_{p-1} }$ to be extended this way we define an 
approximation to the homotopy operator in Eq. \ref{eqchechcontrhom}:

\begin{equation*}
  h_\varphi^{p,q} : \check{C}^p ( \mathfrak{U}, A^q) \rightarrow  
\check{C}^{p-1} ( \mathfrak{U}, A^q), \quad 
h_\varphi^{p,q} ( f )_{i_0 \ldots i_{p-1}} = 
\sum_{i \in I} \varphi_{q,g} f_{i i_0 \ldots i_{p-1}} 
\nonumber
\end{equation*} 
The homomorphisms map continuous cochains to continuous cochains and smooth
cochains to smooth cochains by construction. In addition we observe:

\begin{lemma} \label{lemappcechcontrhom}
For any open covering $\mathfrak{U}:=\{ U_i \mid i \in I\}$ of a smooth
manifold $M$ and set of smooth $R$-valued functions 
$\{\varphi_{q,i} \mid i \in I \}$ with locally finite supports contained in 
$U_i^{q+1}$ respectively the homomorphisms $h_\varphi^{p,q}$ satisfy the equation
\begin{equation} \label{eqappcechcontrhom}
  \delta h_\varphi^{p,q} (f) + h_\varphi^{p+1,q} (\delta f) = 
\sum_i \varphi_{q,i} f 
\end{equation}  
for all cochains $f \in \check{C}^p (\mathfrak{U} , A^q )$.
\end{lemma}

\begin{proof}
For any cochain $f \in \check{C}^p (\mathfrak{U};A^q)$ of bidegree $(p,q)$ 
the horizontal coboundary of $h_\varphi^{p,q} (f)$ computes to 
\begin{eqnarray*}
  (\delta h_\varphi^{p,q} (f))_{i_0 \ldots i_p} (\vec{x}) & = & 
\sum_{k=0}^p (-1)^k h_\varphi^{p,q} (f)_{i_0 \ldots \hat{i}_k \ldots i_p} (\vec{x})\\
& = & \sum_{k=0}^p (-1)^k \sum_{i} 
\varphi_{q,i} (\vec{x}) f_{i i_0 \ldots \hat{i}_k \ldots i_{p-1}} (\vec{x})\\
& = &  \sum_{i} \varphi_{q,i} (\vec{x}) \sum_{k=0}^p (-1)^k 
f_{i i_0 \ldots \hat{i}_k \ldots i_{p-1}} (\vec{x})\\
& = & \sum_i \varphi_{q,i} (\vec{x} ) \left[ f_{i_0 \ldots i_p} (\vec{x}) - 
(\delta f)_{i i_0 \ldots i_p} (\vec{x}) \right] \\
& = & \sum_i \varphi_{q,i} ( \vec{x}) f_{i_0 \ldots i_p} (\vec{x}) - 
h_\varphi^{p+1,q} (\delta f)_{i i_0 \ldots i_p} (\vec{x}) \, ,
\end{eqnarray*}
which is the stated equality.
\end{proof}

\begin{proposition} \label{lemappcechcontrhomres}
For any open identity neighbourhood $U$ of a Lie group $G$ for which all
finite products are smoothly $R$-paracompact, there exists an open identity 
neighbourhood $V \subseteq U$ and homomorphisms 
$h^{p,q} : \check{C}^p ( \mathfrak{U}_U, A^q) \rightarrow  
\check{C}^{p-1} ( \mathfrak{U}_U , A^q)$ satisfying the equation
\begin{equation*} \label{eqappcechcontrhomres}
\res_{\mathfrak{U}_V ,\mathfrak{U}_U}^{p,q} 
\left[ \delta h^{p,q} + h^{p+1,q} \delta \right]
= \res_{\mathfrak{U}_V ,\mathfrak{U}_U}^{p,q} \,
\end{equation*}  
and which leave the sub-rows $\check{C}^* ( \mathfrak{U}_U , A_s^q)$ and 
$\check{C}^* ( \mathfrak{U}_U , A_c^q)$ invariant. In particular the colimit
double complex 
$\colim_{U \in \mathcal{U}_1} \check{C}^* ( \mathfrak{U}_U ;A_s^*)$ is row-exact.
\end{proposition}

\begin{proof}
For any open identity neighbourhood $U$ of a Lie group $G$ for which all
finite products are smoothly $R$-paracompact Lemma \ref{foropenuexistsvandphi} 
shows the existence of an open identity neighbourhood $V \subseteq U$ and 
smooth $R$-valued functions $\{\varphi_{q,g} \mid G \in G\}$ with locally 
finite supports in $(gU) \times \cdots \times (gU)$ respectively such that 
the restriction of each function $\varphi_q = \sum_{g \in G} \varphi_{q,g}$ to 
$\mathfrak{V}[q]$ is the constant function $1$. For $h^{p,q} =
h_\varphi^{p,q}$ the stated equality now follows from Lemma 
\ref{lemappcechcontrhom}.
\end{proof}

\begin{corollary} \label{corfinprodpcthenjsindiso}
For any open neighbourhood filterbase $\mathcal{U}_1$ of a Lie group $G$ whose 
finite products are smoothly $R$-paracompact the morphisms 
$i_s^*$ induce an isomorphism 
$\colim_{U \in \mathcal{U}_1} H_s (\mathfrak{U}_U ;V) \cong 
\colim_{U \in \mathcal{U}_1} H (\tot\check{C}^* (\mathfrak{U}_U ;A_s^*))$.  
\end{corollary}

Summarising the preceding observations for Lie groups we have shown:

\begin{theorem} \label{colimspclg}
 For any open neighbourhood filterbase $\mathcal{U}_1$ of a Lie group $G$ whose 
finite products are smoothly $R$-paracompact the morphisms 
$A_s^* (\mathfrak{U}_U ;V) \hookrightarrow A^* (\mathfrak{U}_U ;V)$ and 
$C (\lambda_{\mathfrak{U}_U}^* ,V) :  A_s^* (\mathfrak{U}_U ; V) 
\rightarrow S^* (\mathfrak{U}_U ;V)$ for all $U \in \mathcal{U}_1$ induce 
isomorphisms  
\begin{equation*}
 \colim_{U \in \mathcal{U}_1} H_s (\mathfrak{U}_U ;V) \cong 
\colim_{U \in \mathcal{U}_1} H (\mathfrak{U}_U ;V) \cong H_{sing} (G;V)
\end{equation*}
in cohomology.
\end{theorem}

\begin{proof}
For every Lie group $G$ with open identity neighbourhood filterbase
$\mathcal{U}_1$ the inclusions 
$A_s^* (\mathfrak{U}_U ;V) 
\hookrightarrow A^* (\mathfrak{U}_U ;V)$ and 
$\tot \check{C}^* (\mathfrak{U}_U  ; A_s^*) \hookrightarrow 
\tot \check{C}^* (\mathfrak{U}_U ; A^*)$ of cochain complexes lead to the 
commutative diagram
\begin{equation*}
  \xymatrix{
\colim_{U \in \mathcal{U}_1} A_s^* (\mathfrak{U}_U ;V) \ar[r]^{i_s^*} \ar[d] & 
\colim_{U \in \mathcal{U}_1} \tot \check{C}^* ( \mathfrak{U}_U , A_s^*) \ar[d] & 
\colim_{U \in \mathcal{U}_1} \check{C} (\mathfrak{U}_U ;V) \ar[l]_{j_s^*} \ar@{=}[d] \\
\colim_{U \in \mathcal{U}_1} A^* (\mathfrak{U}_U ;V) \ar[r]^{i^*} 
& 
\colim_{U \in \mathcal{U}_1} \tot \check{C}^* ( \mathfrak{U}_U , A^*) 
& 
\colim_{U \in \mathcal{U}_1} \check{C} (\mathfrak{U}_U ;V) \ar[l]_{j^*} 
}
\end{equation*}
of cochain complexes. The morphisms $j_s^*$ and $j^*$ on the right
hand side induce isomorphisms in cohomology, hence the inclusions of total
complexes also induce isomorphisms in cohomology. The inclusions $i_s^*$ and 
$i^*$ also induce isomorphisms in cohomology 
(by Corollary \ref{corfinprodpcthenjsindiso} and Theorem \ref{colimnbhf}),
which proves the first isomorphism. The second isomorphism also follows from 
Theorem \ref{colimnbhf}).
\end{proof}

\section{Loop Contractible Coefficients}
\label{seclcc}

In this section we consider a different class of coefficient groups $V$ and 
derive results analogous to previously obtained ones. 
To obtain exact rows in the 
double complex $\check{C}^* (\mathfrak{U}, A_c^*)$ we again impose a 
restriction on the coefficient group $V$; however this time it is an algebraic
topological one:

\begin{definition}
  A (semi-)topological group $G$ is called \emph{loop contractible}, 
if there exists a contraction $\Phi : G \times I \rightarrow G$ to the 
identity such that 
$\Phi_t : G \rightarrow G, g \mapsto \Phi (g,t)$ is a homomorphism of
(semi-)topological groups for all $t \in I$.
\end{definition}

\begin{example}
  Any topological vector space $V$ is loop contractible via 
$\Phi (v,t)=t \cdot v$.
\end{example}

\begin{example}
  The path group $PG=C ( (I,\{ 0 \}),(G, \{e \}))$ of based paths of a 
topological group $G$ is loop contractible via 
$\Phi_{PG} (\gamma , s) (t):= \gamma (st)$.  
\end{example}

\begin{remark}
  A topological group $G$ is loop contractible if and only if the extension 
$\Omega G \hookrightarrow PG \twoheadrightarrow G$ is a semi-direct product: If
$\Phi_G$ is a loop contraction of $G$, then the group homomorphism 
$s:G \rightarrow PG, s (g)(t)=\Phi_G (g,t)$ is a right inverse to the 
evaluation $\mathrm{ev}_1 :PG \rightarrow G$ at $1$; conversely, if such a right
inverse $s$ exists, then the homotopy given by 
$\Phi_G (g,t) := \mathrm{ev}_1 \Phi_{PG} (s(g))(t)$ 
is a loop contraction of $G$.
\end{remark}

\begin{example}
For a topological group $G$ the geometric realisation $EG:=| G^{*+1} |$ of 
the simplicial space $G^{*+1}$ is a semi-topological group. As observed in
\cite{MB78} the elements in $EG$ can be identified with the step functions 
$f:[0,1) \rightarrow G$ which are continuous from the right and the
multiplication in $EG$ is given by the pointwise multiplication of these step
functions.  
The natural contraction of $EG$ is explicitely given by  
\begin{equation*}
  \Phi (f,s) (t):=
  \begin{cases}
    e & \text{if} \, t < s \\
f (t) & \text{if} \, s \leq t
  \end{cases}
\end{equation*}
(cf. the contraction in \cite[Section 11.2]{F10} in \cite[p. 214]
{MB78}). This is a loop contraction of the semi-topological group $EG$. 
\end{example}

If the abelian topological coefficient group $V$ is loop contractible then 
one can generalise the classical construction of the row contractions in 
Proposition \ref{propchechcontrhomcont}. 
For this purpose we 
consider the singular semi-simplicial space $C (\varDelta , V)$ of $V$ and the 
vertex morphism $\lambda_V : C (\varDelta , V) \rightarrow V^{*+1}$ of 
semi-simplicial spaces, which assigns to each singular $n$-simplex 
$\tau : \varDelta^n \rightarrow V$ its ordered set of vertices 
$(\tau ( \vec{e}_0) ,\ldots, \tau (\vec{e}_n))$. 
It has been shown in \cite[Chapter 3]{F10}, that there exists a right inverse 
$\widehat{\sigma}$ 
to the vertex morphism $\lambda_V$. 
The  construction is as follows: Let $\Phi : V\times I \rightarrow V$ be a
loop contraction and consider the continuous map
\begin{equation*}
F: V \times V \times I \rightarrow V,  
\quad (v_0,v_1,t) \mapsto v_0 + \Phi (v_1 -v_0 ,t) \, ,
\end{equation*}
then start by setting $\hat{\sigma}_0 (v)(t_0) =v$ and by inductively defining
the functions $\widehat{\sigma}_n$ via  
\begin{equation} \label{defsigman}
  \widehat{\sigma}_{n+1} (\vec{v} )(\vec{t}):=
  \begin{cases}
    v_0 & \text{if } \,  t_0 =1 \\
F \left( v_0 , \widehat{\sigma}_n ( v_1 ,\ldots,v_{n+1} ) 
\left( \frac{t_1}{ 1 -t_0}, \ldots, \frac{t_{n+1}}{ 1 -t_0}  \right) , 
t_0 \right) & \text{if } \, t_0 \neq 0
  \end{cases}
\end{equation}
We equip the spaces $C(\varDelta^n ,V)$ with the compact-open topology. 
The crucial properties of the maps 
$\widehat{\sigma}_n : V^{n+1} \rightarrow (\varDelta^n, V)$ we rely on are:

\begin{proposition} \label{sigmaismorph}
  The map $\widehat{\sigma} : V^{*+1} \rightarrow C (\varDelta^* , V)$ is a
  morphism of semi-simplicial topological spaces which is a right inverse to the
  vertex morphism $\lambda_V$ and all the adjoint functions 
$\sigma_n : V^{n+1} \times \varDelta^n \rightarrow V$ are continuous. In
addition for all $v \in V$ the singular $n$-simplices 
$\widehat{\sigma}_n (v,\ldots,v)$ are the constant maps 
$\varDelta^n \rightarrow \{ v \}$.
\end{proposition}

\begin{proof}
The continuity of the maps 
$\sigma_n : V^{n+1} \times \varDelta^n \rightarrow V$ is shown in 
\cite[Lemma 3.0.69]{F10}, the fact that $\widehat{\sigma}$ is a morphism of
semi-simplicial spaces is the content of \cite[Lemma 3.0.70]{F10}; 
the proofs presented there carry over in verbatim. 
The last statement is a consequence of the inductive definition 
\ref{defsigman} of the functions $\widehat{\sigma}_n$.
\end{proof}

\begin{lemma} \label{loopthenhom}
  If $\Phi :V \times I \rightarrow V$ is a loop contraction, then 
$\widehat{\sigma} : V^{*+1} \rightarrow C (\varDelta^* , V)$ is a morphism of
semi-simplicial abelian topological groups.
\end{lemma}

\begin{proof}
It is to show that each map $\widehat{\sigma}_n : V^{n+1} \rightarrow C
(\varDelta^n , V)$ is a group homomorphism. 
  This is proved by induction. The functions $\widehat{\sigma}_0$ are group
  homomorphisms by definition. 
Moreover, since $\Phi$ is a loop contraction, the function $F : V \times V
\times I \rightarrow V$ is additive in $V \times V$. Now assume that the 
function $\widehat{\sigma}_n$ is a group homomorphism and let 
$\vec{v} , \vec{w} \in V^{n+2}$ be given. The singular $(n+1)$-simplex 
$\widehat{\sigma}_{n+1} (\vec{v} + \vec{w})$ takes the value $v_0 + w_0$ at 
$t_0=1$. For $t_0 \neq 0$ its value is given by 
\begin{eqnarray*}
\lefteqn{\widehat{\sigma}_{n+1} (\vec{v} + \vec{w}) (\vec{t})= } 
\hspace{30pt}\\
& = & F \left( v_0 +w_0 , \widehat{\sigma}_n ( v_1+w_1 ,\ldots,v_{n+1}+w_{n+1}) 
\left( \frac{t_1}{ 1 -t_0}, \ldots, \frac{t_{n+1}}{ 1 -t_0}  \right) , 
t_0 \right) \\
& = & F \left( v_0 , \widehat{\sigma}_n ( v_1 , \ldots,v_{n+1} ) 
\left( \frac{t_1}{ 1 -t_0}, \ldots, \frac{t_{n+1}}{ 1 -t_0}  \right) , 
t_0 \right) \\
& & + F \left( w_0 , \widehat{\sigma}_n ( w_1 ,\ldots,w_{n+1}) 
\left( \frac{t_1}{ 1 -t_0}, \ldots, \frac{t_{n+1}}{ 1 -t_0}  \right) , 
t_0 \right) \\ 
& = & \widehat{\sigma}_{n+1} (\vec{v}) (\vec{t}) + 
\widehat{\sigma}_{n+1} (\vec{w}) (\vec{t})
\end{eqnarray*}
which completes the inductive step. 
\end{proof}

From now on we assume the coefficient group $V$ to be loop contractible with
loop contraction $\Phi : V \times I \rightarrow V$, which gives rise to a
morphism $\widehat{\sigma} : V^{*+1} \rightarrow C (\varDelta^* , V)$ of 
semi-simplicial abelian topological groups that is a right inverse
to $\lambda_V$. The above observations enable us to replace 
the linear combination $\sum_i \varphi_{q,i} f_{i i_0 \ldots i_p}$ of functions 
in Proposition \ref{propchechcontrhom} by the values of the singular 
$n$-simplices $\widehat{\sigma}_n (f_{\alpha_0 i_0 \ldots i_p},\ldots
,f_{\alpha_n i_0 \ldots i_p})$ at 
$( \varphi_{q, \alpha_0} , \ldots \varphi_{q, \alpha_n})$ for certain indices 
$\alpha_0,\ldots,\alpha_n \in I$. For this purpose we first observe:

\begin{lemma} \label{hpfcont}
For any partition of unity $\{ \varphi_{q,i} \mid i \in I\}$ 
subordinate to the open cover $\{U_i^{q+1} \mid i \in I \}$ of 
$\mathfrak{U}[q]$ and $p$-cochain $f \in \check{C}^p (\mathfrak{U} , A_c^q)$ 
the maps 
  \begin{equation} \label{predefhp}
  U_{i_0 \ldots i_{p-1}}^{q+1} \rightarrow V , \quad 
x \mapsto 
\sigma_n \circ ( f_{\alpha_0 i_0 \ldots i_{p-1}} , \ldots , f_{\alpha_n i_0
  \ldots i_{p-1}} , \varphi_{q, \alpha_0} , \ldots , \varphi_{q, \alpha_n} ) 
(\vec{x})
\end{equation}
-- where for each $x \in X$ the indices $\alpha_0 < \cdots < \alpha_n$ are
those for which $\varphi_{q, \alpha_i}^{-1} (\mathbb{R} \setminus \{ 0 \})$ 
contains $x$ -- are continuous.
\end{lemma}

\begin{proof}
It suffices to show that each point $\vec{x} \in \mathfrak{U}[q]$ has a 
neighbourhood on which 
the functions defined in \ref{predefhp} are continuous. This is a consequence 
of the fact that $\hat{\sigma} : V^{*+1} \rightarrow C (\varDelta , V)$ is a 
morphism of semi-simplicial spaces: Since the supports of the functions 
$\varphi_{q,i}$ are locally finite, each point $\vec{x} \in \mathfrak{U}[q]$
has a neighbourhood $W$ such that the set 
$I_W :=\{ \alpha' \in I \mid \, 
\varphi_{q,\alpha' } (\vec{x}) \neq 0 \}$ 
is finite. Let $\alpha_0' < \cdots < \alpha_{k}'$ be the ordered set of indices
in $I_W$. The function defined by 
\begin{equation*}
  x \mapsto \sigma_k  ( f_{\alpha_0' i_0 \ldots i_{p-1}} (x) , \ldots , f_{\alpha_k' i_0 \ldots
  i_{p-1}} (x) , \varphi_{q, \alpha_0'} (x) , \ldots , \varphi_{q, \alpha_k'} (x) )
\end{equation*}
is continuous on the open set 
$U_{\alpha_0 ' \ldots \alpha_n ' i_0 \ldots i_{p-1}}^{q+1}$, 
which is a neighbourhood of $\vec{x}$. Let $\alpha_0 < \cdots < \alpha_n$ be 
the ordered set of indices for which 
$\varphi_{q, \alpha_i}^{-1} (\mathbb{R} \setminus \{ 0 \})$ contains
$\vec{x}$; it is a subset of $I_W$. The fact that $\sigma$ is a morphism of 
semi-simplicial spaces implies the equality 
\begin{multline*}
 \sigma_k  ( f_{\alpha_0' i_0 \ldots i_{p-1}} (x) , \ldots , f_{\alpha_k' i_0
   \ldots i_{p-1}} (x) , \varphi_{q, \alpha_1'} (x) , \ldots ,
 \varphi_{q, \alpha_k'} (x) ) = \\ 
\sigma_n  ( f_{\alpha_0 i_0 \ldots i_{p-1}} (x) , \ldots , f_{\alpha_n i_0 \ldots
  i_{p-1}} (x) , \varphi_{q, \alpha_0} (x) , \ldots , \varphi_{q, \alpha_n}
(x) ) \,
\end{multline*}
which shows that the function defined in \ref{predefhp} is continuous on the 
open neighbourhood $U_{\alpha_0 ' \ldots \alpha_n ' i_0 \ldots i_{p-1}}^{q+1}$
of $\vec{x}$. 
\end{proof}

\begin{lemma} \label{hpfcoch}
For any partition of unity $\{ \varphi_{q,i} \mid i \in I\}$ 
subordinate to the open cover $\{U_i^{q+1} \mid i \in I \}$ of 
$\mathfrak{U}[q]$ and $p$-cochain $f \in \check{C}^p (\mathfrak{U} , A_c^q)$ 
the maps defined in \ref{predefhp} form a cochain in 
$\check{C}^{p-1} (\mathfrak{U};A_c^q)$.   
\end{lemma}

\begin{proof}
It is to show that for each $p$-cochain 
$f \in \check{C}^p (\mathfrak{U} , A_c^q)$ and permutation $s$ of 
$\{ i_0 \ldots i_p \}$ the maps defined in \ref{predefhp} satisfy the 
equalities 
\begin{multline*}
\sigma_n  ( f_{\alpha_0 i_0 \ldots i_{p-1}} (x) , \ldots , f_{\alpha_n i_0
 \ldots i_{p-1}} (x) , \varphi_{q, \alpha_0} (x) , \ldots , \varphi_{q, \alpha_n}
(x) ) = \\
\mathrm{sign} (s) 
  \sigma_n ( f_{\alpha_0 i_{s (0)} \ldots i_{s (p-1)} } (x), 
\ldots , f_{\alpha_n i_{s (0)} \ldots i_{s (p-1)} } (x), 
\varphi_{q, \alpha_0} (x) , \ldots , \varphi_{q, \alpha_n} (x) ) \, .
\end{multline*}
This is a consequence of Lemma \ref{loopthenhom}. Thus the assignment in 
\ref{predefhp} defines a cochain in $\check{C}^{p-1} (\mathfrak{U};A_c^q)$.
\end{proof}

\begin{proposition} \label{propchechcontrhomloop}
For any partition of unity $\{ \varphi_{q,i} \mid i \in I\}$ subordinate to 
the open cover $\{U_i^{q+1} \mid i \in I \}$ of $\mathfrak{U}[q]$ the 
homomorphisms 
\begin{eqnarray}\label{eqchechcontrhomloop}
  h^{p,q} : \check{C}^p ( \mathfrak{U}, A^q) & \rightarrow &  
\check{C}^{p-1} ( \mathfrak{U}, A^q) \\
(h^{p,q} f)_{i_0 \ldots i_{p-1}} & = &  
\sigma_n \circ ( f_{\alpha_0 i_0 \ldots i_{p-1}} , \ldots , f_{\alpha_n i_0 \ldots i_{p-1}} , 
\varphi_{q, \alpha_0} , \ldots , \varphi_{q, \alpha_n} ) \, , \nonumber
\end{eqnarray}
-- where for each $x \in X$ the indices $\alpha_0 < \cdots < \alpha_n$ are
those satisfying $\varphi_{q, \alpha_i} (\vec{x}) \neq 0$ --
form a contraction of the augmented row 
$A^q (\mathfrak{U};V) \hookrightarrow \check{C}^* ( \mathfrak{U}, A^q)$ which
restricts to a row contraction of the augmented sub-complex 
$A_c^q (\mathfrak{U};V) \hookrightarrow \check{C}^* ( \mathfrak{U}, A_c^q)$.
\end{proposition}

\begin{proof}
The maps $h^{p,q}$ are homomorphisms of abelian groups by Lemma
\ref{loopthenhom} and map the subgroups $\check{C}^* ( \mathfrak{U}, A_c^q)$
of continuous cochains into each other by Lemma \ref{hpfcont}. 
Consider a point $\vec{x} \in \mathfrak{U}[q]$ and let 
$\alpha_0,\ldots, \alpha_n$ be the ordered set of indices in $I$ for which 
$\varphi_{q,i}^{-1} ( (0,1])$ contains the point $\vec{x}$. The evaluation of 
$(h^{p+1,q} \delta f)_{i_0 \ldots i_p}$ at $\vec{x}$ computes to
\begin{eqnarray*}
\lefteqn{(h^{p+1} \delta f)_{i_0 \ldots i_p} (\vec{x}) =  
\sigma_n  ( (\delta f)_{\alpha_0 i_0 \ldots i_p} (x) , \ldots , (\delta
f)_{\alpha_n  i_0 \ldots i_p} (x) , \varphi_{q, \alpha_0} (x) , \ldots,
\varphi_{q, \alpha_n} (\vec{x}) ) } \hspace{46pt} \\
& = & \sigma_n  ( f_{i_0 \ldots i_p} (x) , \ldots , f_{i_0 \ldots i_p} (\vec{x}) , 
\varphi_{q, \alpha_0} (x) , \ldots, \varphi_{q, \alpha_n} (\vec{x}) ) \\ 
& & - 
\sum_k (-1)^k \sigma_n  ( f_{\alpha_0 i_0 \ldots \hat{i}_k \ldots i_p} , 
\ldots  f_{\alpha_n i_0 \ldots \hat{i}_k \ldots i_p} , 
\varphi_{q, \alpha_0} (x) , \ldots, \varphi_{q, \alpha_n} (\vec{x}) ) \\
& =& 
f_{i_0 \ldots i_p} ( x )  - (\delta h^p f)_{i_0 \ldots i_p} (\vec{x}) \, .
\end{eqnarray*}
Thus the homomorphisms $h^{p,q}$ form a row contraction of 
$A^q (\mathfrak{U};A) \hookrightarrow \check{C}^* ( \mathfrak{U}, A^q)$, which 
restricts to a row contraction of the sub-complex 
$A_c^q (\mathfrak{U};A_c) \hookrightarrow \check{C}^* ( \mathfrak{U}, A_c^q)$.
\end{proof}

\begin{corollary} \label{uqnumicindisoloop}
For any open covering $\mathfrak{U}=\{ U_i \mid i \in I \}$ of a topological 
space $X$ for which the coverings $\{ U_i^{q+1} \mid i \in I \}$ of the 
spaces $\mathfrak{U}[q]$ are numerable the homomorphism 
$i_c^*:A_c^* (\mathfrak{U};V)\rightarrow \tot\check{C}^* (\mathfrak{U},A_c^*)$ 
induces an isomorphism in cohomology.
\end{corollary}
Recalling the contractibility condition imposed on the coefficient group $V$ 
we proceed to show:

\begin{theorem}\label{isocontuqnumloop}
For any loop contractible abelian topological group $V$ and
open covering $\mathfrak{U}$ of a topological space $X$ for which 
each covering $\{ U_i^{q+1} \mid i \in I \}$ of $\mathfrak{U}[q]$ is numerable 
the inclusion $A_c^* (\mathfrak{U};V) \hookrightarrow A^* (\mathfrak{U};V)$ 
induces an isomorphism in cohomology and the cohomologies 
$\check{H} (\mathfrak{U};V)$, $H_c (\mathfrak{U};V)$ and $H (\mathfrak{U};V)$ 
are isomorphic.
\end{theorem}

\begin{proof}
  The proof is analogous to that of Theorem \ref{isocontuqnum}.
\end{proof}

In this case the \v{C}ech Cohomology $\check{H} (\mathfrak{U};V)$ for the
covering $\mathfrak{U}$ of X can be either computed from the complex 
$A_c^* (\mathfrak{U};V)$ of continuous $\mathfrak{U}$-local cochains or from 
from the complex $A^* (\mathfrak{U};V)$ of $\mathfrak{U}$-local cochains.

\begin{corollary} \label{covbycozloop}
For any loop contractible abelian topological group $V$, generalised 
partition of unity $\{ \varphi_i \mid i \in I\}$ on 
$X$ and $\mathfrak{U}:=\{ \varphi_i^{-1} (R \setminus \{ 0\}) \mid i\in I\}$ 
the inclusion $A_c^* (\mathfrak{U};V) \hookrightarrow A^* (\mathfrak{U};V)$ 
induces an isomorphism in cohomology and the cohomologies 
$\check{H} (\mathfrak{U};V)$, $H_c (\mathfrak{U};V)$ and $H (\mathfrak{U};V)$ 
are isomorphic.
\end{corollary}

\begin{proof}
  The proof is analogous to that of Corollary \ref{covbycoz}.
\end{proof}

Passing to the colimit over all numerable coverings yields the classical 
results: 

\begin{corollary}
  For any topological space $X$ and loop contractible abelian topological 
group $V$ the \v{C}ech cohomology 
$\check{H} (X;V)$ w.r.t. numerable coverings and the continuous 
Alexander-Spanier cohomology $H_{AS,c} (X;V)$ w.r.t. numerable coverings are 
isomorphic. 
\end{corollary}

\begin{corollary}
  For any paracompact topological space $X$ and loop contractible coefficient 
group $V$ the \v{C}ech cohomology 
$\check{H} (X;V)$ and the continuous Alexander-Spanier 
cohomology $H_{AS,c} (X;V)$ are isomorphic. 
\end{corollary}

\begin{example}
  If a paracompact space $X$ has trivial \v{C}ech cohomology $\check{H} (X;V)$ 
(e.g. if $X$ is contractible) and $V$ is loop contractible, then the
  continuous Alexander-Spanier cohomology $H_{AS} (X;V)$ is trivial as well.
\end{example}

As we did before, we apply these observations to uniform spaces $X$ with open 
coverings of the form $\mathfrak{U}_U :=\{ U[x] \mid \, x \in X\}$, where 
$U$ is an open entourage of the diagonal in $X \times X$.

\begin{proposition} \label{ubypmloop}
  If $d: X \times X \rightarrow \mathbb{R}$ is a continuous pseudometric on 
$X$ then for each $\epsilon > 0$, covering 
$\mathfrak{U}=\{ B_d (x,\epsilon) \mid x \in X \}$ of $X$ by open 
$\epsilon$-balls and loop contractible abelian topological group $V$ 
the inclusion 
$A_c^* (\mathfrak{U};V) \hookrightarrow A^* (\mathfrak{U};V)$ induces an
isomorphism in cohomology and the cohomologies 
$\check{H} (\mathfrak{U};V)$, $H_c (\mathfrak{U};V)$ and $H (\mathfrak{U};V)$ 
are isomorphic.
\end{proposition}

\begin{proof}
  The proof is analogous to that of Proposition \ref{ubypm}.
\end{proof}

\begin{example}
If $\mathfrak{U}=\{ B (x,\epsilon) \mid x \in X \}$ is an open covering of a 
finite dimensional Riemannian manifold $M$ by open $\epsilon$-balls and 
the coefficient group $V$ is loop contractible, 
then the cohomology $H_c (\mathfrak{U};V)$ of the complex 
$A_c^* (\mathfrak{U};V)$ is isomorphic to cohomology $H (\mathfrak{U} ;V)$ and 
to the \v{C}ech and singular cohomologies of $M$ (cf. Ex \ref{exugeodconv}). 
If $M$ is an infinite dimensional Riemannian manifold one has to require the 
local existence of geodesics.  
\end{example}

\begin{corollary}
  For any open entourage $U$ of a uniform space $X$ 
and loop contractible abelian topological group $V$ 
the inclusion 
$A_c^* (\mathfrak{U}_U;V)\hookrightarrow A^* (\mathfrak{U}_U;V)$ induces an 
isomorphism in cohomology and the cohomologies 
$\check{H} (\mathfrak{U}_U;V)$, $H_c (\mathfrak{U}_U;V)$ and 
$H (\mathfrak{U}_U;V)$ are isomorphic.
\end{corollary}

For topological groups with open coverings of the 
form $\mathfrak{U}_U :=\{ gU \mid \, g \in G\}$, where $U$ is an open identity 
neighbourhood in $G$ we observe:

\begin{corollary} \label{localfortopgrpsandvsvloop}
For any open identity neighbourhood $U$ of a topological group $G$ and 
loop contractible coefficient group $V$ the 
inclusion $A_c^* (\mathfrak{U}_U ;V) \hookrightarrow A^* (\mathfrak{U}_U ;V)$
induces an isomorphism in cohomology and the cohomologies 
$\check{H} (\mathfrak{U}_U;V)$, $H_c (\mathfrak{U}_U;V)$ and 
$H (\mathfrak{U}_U;V)$ are isomorphic.
\end{corollary}

Combining these results with those concerning singular cohomology 
(obtained in Section \ref{seccas}) we observe:

\begin{proposition} \label{propuqnumsingloop}
For any loop contractible abelian group $V$ and open covering 
$\mathfrak{U}$ of a space $X$ for which each set $U_{i_0 \ldots i_p}$ is 
$V$-acyclic and each covering 
$\{ U_i^{q+1} \mid i \in I \}$ of $\mathfrak{U}[q]$ is numerable the 
homomorphism $C (\lambda_\mathfrak{U}^* ,V) : 
A_c^* (\mathfrak{U};V) \rightarrow S^* (X,\mathfrak{U};V)$ 
induces an isomorphism $H_c (\mathfrak{U};V) \cong H_{sing} (X;V)$ in 
cohomology and the following diagram is commutative:
\begin{equation*}
  \xymatrix{
H_c (\mathfrak{U};V) \ar[r]^{H (i)}_\cong 
\ar[d]_{H (C (\lambda_\mathfrak{U}^*;V))}^\cong  &
H (\mathfrak{U};V) \ar[r]^{H (j)^{-1} H(i)}_\cong 
\ar[d]_{H (C (\lambda_\mathfrak{U}^* ;V))}^\cong & 
\check{H} (\mathfrak{U};V) \ar@{=}[d] \\
H_{sing} (\mathfrak{U};V) \ar@{=}[r] & 
H_{sing} (\mathfrak{U};V) \ar[r]^\cong_{H (j)^{-1} H(i)} & 
\check{H} (\mathfrak{U};V) 
}
\end{equation*}
In particular the \v{C}ech and the continuous $\mathfrak{U}$-local 
cohomology do not depend on the open cover $\mathfrak{U}$ subject to the 
above conditions chosen.
\end{proposition}

\begin{corollary}
 For loop contractible coefficient groups $V$ and any open covering 
$\mathfrak{U}$ of a topological space $X$ 
for which each set $U_{i_0 \ldots i_p}$ is $V$-acyclic and each covering 
$\{ U_i^{q+1} \mid i \in I \}$ of $\mathfrak{U}[q]$ is numerable the 
singular cohomology $H_{sing} (X;V)$ and the \v{C}ech cohomology 
$\check{H} (\mathfrak{U};V)$ for the covering $\mathfrak{U}$ 
can be computed from the complex $A_c^* (\mathfrak{U};V)$ of 
continuous $\mathfrak{U}$-local cochains. 
\end{corollary}

\begin{example}
For any loop contractible coefficient group $V$ and 'good' cover 
$\mathfrak{U}$ of a topological space $X$ for which 
each covering $\{ U_i^{q+1} \mid i \in I \}$ of $\mathfrak{U}[q]$ is numerable
the morphism $ 
A_c^* (\mathfrak{U};V) \rightarrow S^* (X,\mathfrak{U};V)$ of cochain complexes
induces an isomorphism in cohomology and the cohomologies 
$\check{H} (\mathfrak{U};V)$, $H (\mathfrak{U};V)$, $H_c (\mathfrak{U};V)$ 
and $H_{sing} (X;V)$ are isomorphic.
\end{example}

\begin{lemma}
If $V$ is loop contractible and the open coverings $\mathfrak{U}$ of a 
topological space $X$ for which the sets $U_{i_0 \ldots i_p}$ are $V$-acyclic 
and  each covering $\{ U_i^{q+1}
\mid i \in I \}$ of $\mathfrak{U}[q]$ is numerable are cofinal in all 
open coverings, then for each such covering $\mathfrak{U}$ the 
the cohomology $H_c (\mathfrak{U};V)$ of continuous $\mathfrak{U}$-local 
cochains coincides with the continuous Alexander-Spanier cohomology 
$H_{AS,c} (X;V)$ of $X$. In particular the directed system 
$H_c (\mathfrak{U};V)$ of abelian groups is co-Mittag-Leffler.
\end{lemma}

\begin{proof}
  The proof is analogous to that of Lemma \ref{vacnumarecofin}.
\end{proof}

\begin{example}
If $\mathfrak{U}$ is an open covering of a finite dimensional Riemannian 
manifold $M$ by geodetically convex sets and $V$ is loop contractible, 
then the cohomology of the complex 
$A_c^* (\mathfrak{U};V)$ is isomorphic to the \v{C}ech cohomology $\check{H}
(\mathfrak{U};V)$ and to the singular cohomology $H_{sing} (M;V)$ of $M$. 
If $M$ is an infinite dimensional Riemannian manifold one has to require the 
local existence of geodesics for this argument to be applicable.
\end{example}

\begin{example}
  If $G$ is a Hilbert Lie group, $U$ a geodetically convex identity 
neighbourhood of $G$ and $V$ is loop contractible, 
then the cohomology of the complex 
$A_c^* (\mathfrak{U}_U;V)$ is isomorphic to \v{C}ech cohomology $\check{H}
(\mathfrak{U};V)$ and to the singular cohomology $H_{sing} (G;V)$ of $G$.
\end{example}

As observed before, one can obtain a similar result for locally contractible
topological groups without acyclicity condition on the open coverings:

\begin{theorem} \label{colimnbhfcontloop}
  For any locally contractible group $G$ with open identity neighbourhood 
filterbase $\mathcal{U}_1$ and loop contractible $V$ the morphisms 
$A_c^* (\mathfrak{U}_U ;V) \hookrightarrow A^* (\mathfrak{U}_U ;V)$ and 
$C (\lambda_{\mathfrak{U}_U}^* ,V) :  A_c^* (\mathfrak{U}_U ; V) 
\rightarrow S^* (\mathfrak{U}_U ; V)$ for all $U \in \mathcal{U}_1$ induce 
isomorphisms  
\begin{equation*}
 \colim_{U \in \mathcal{U}_1} H_c (\mathfrak{U}_U ;V) \cong 
\colim_{U \in \mathcal{U}_1} H (\mathfrak{U}_U ;V) \cong H_{sing} (G;V) \, .
\end{equation*}
in cohomology. 
\end{theorem}

\begin{proof}
  The proof is analogous to that of Theorem \ref{colimnbhfcontloop}.
\end{proof}

\begin{corollary}
For Lie groups $G$ with open identity neighbourhood filter base 
$\mathcal{U}_1$ and loop contractible coefficient groups $V$ 
the cohomologies $\colim_{U \in \mathcal{U}_1} H_c (\mathfrak{U}_U ;V)$, 
$\colim_{U \in \mathcal{U}_1} H (\mathfrak{U}_U ;V)$, 
$\colim_{U \in \mathcal{U}_1} \check{H} (\mathfrak{U}_U;V)$ and 
$H_{sing} (G;V)$ coincide. 
\end{corollary}

\section{$k$-Spaces and Loop Contractible Coefficients}
\label{seclcck}

In this section we work in the category of $k$-spaces and derive results 
analogous to previously obtained ones.  
To obtain exact rows in the double complex 
$\check{C}^* (\mathfrak{U}, A_{kc}^*)$ we again impose the restriction of loop
contractibility.  

\begin{example}
  Any $k$-vector space $V$ is loop contractible via $\Phi (v,t)=t \cdot v$.
\end{example}

\begin{example}
The path $k$-group 
$\fktop PG= \fktop C ( (I,\{ 0 \}),(G, \{e \}))$ of based paths of a 
$k$-group $G$ is loop contractible via 
$\Phi_{PG} (\gamma , s) (t):= \gamma (st)$.  
\end{example}

\begin{remark}
  A $k$-group $G$ is loop contractible if and only if the extension 
$\fktop \Omega G \hookrightarrow \fktop PG \twoheadrightarrow G$ is a 
semi-direct product: If $\Phi_G$ is a loop contraction of $G$, then the group 
homomorphism $s:G \rightarrow \fktop PG, s (g)(t)=\Phi_G (g,t)$ is a right 
inverse to the evaluation $\mathrm{ev}_1 :PG \rightarrow G$ at $1$; 
conversely, if such a right inverse $s$ exists, then the homotopy given by 
$\Phi_G (g,t) := \mathrm{ev}_1 \Phi_{PG} (s(g))(t)$ 
is a loop contraction of $G$.
\end{remark}

\begin{example}
For a $k$-group $G$ the geometric realisation 
$\fktop EG=| \fktop G^{*+1} |$ of the simplicial
$k$-group $G^{*+1}$ is a $k$-group. As observed in \cite{MB78} the elements 
in $\fktop EG$ can be identified with the step functions 
$f:[0,1) \rightarrow G$ which are continuous from the right and the 
multiplication in $EG$ is given by the pointwise multiplication of these step 
functions.  
The natural contraction of $EG$ is explicitely given by  
\begin{equation*}
  \Phi (f,s) (t):=
  \begin{cases}
    e & \text{if} \, t < s \\
f (t) & \text{if} \, s \leq t
  \end{cases}
\end{equation*}
(cf. the contraction in \cite[Section 11.2]{F10} in \cite[p. 214]
{MB78}). This is a loop contraction of the $k$-group $EG$. 
\end{example}

If the coefficient $k$-group $V$ is loop contractible then one can transfer
the construction of the semi-simplicial morphism 
$\widehat{\sigma} : V^{*+1} \rightarrow C (\varDelta^* , V)$ constructed in
Section \ref{seclcc} to the category of $k$-spaces. This is done by replacing
products in $\tops$ by products in $\ktop$ in the inductive definition of 
$\widehat{\sigma}$. We denote the so obtained map 
$\fktop V^{*+1} \rightarrow \fktop C (\varDelta,X)$ by 
$\fktop \widehat{\sigma}$.

\begin{proposition}
The map 
$\fktop \widehat{\sigma} :\fktop V^{*+1} \rightarrow \fktop C(\varDelta^*,V)$ 
is a morphism of semi-simplicial $k$-spaces which is a right inverse to the
vertex morphism $\fktop \lambda_V$ and all the adjoint functions 
$\fktop \sigma_n : \fktop V^{n+1} \times \varDelta^n \rightarrow V$ are 
continuous. In addition for all $v \in V$ the singular $n$-simplices 
$\widehat{\sigma}_n (v,\ldots,v)$ are the constant maps 
$\varDelta^n \rightarrow \{ v \}$.
\end{proposition}

\begin{proof}
The proof is analogous to that of Proposition \ref{sigmaismorph}.
\end{proof}

\begin{lemma} \label{loopthenhomk}
  If $\Phi :V \times I \rightarrow V$ is a loop contraction, then 
$\widehat{\sigma} : \fktop V^{*+1} \rightarrow \fktop C (\varDelta^* , V)$ 
is a morphism of semi-simplicial abelian topological groups.
\end{lemma}

\begin{proof}
The proof is analogous to that of Lemma \ref{loopthenhomk}.
\end{proof}

From now on we assume the coefficient $k$-group $V$ to be loop contractible 
with loop contraction $\Phi : V \times I \rightarrow V$, which gives rise to a
morphism 
$\fktop \widehat{\sigma} :\fktop V^{*+1} \rightarrow \fktop C(\varDelta^*,V)$ 
of semi-simplicial abelian $k$-groups that is a right inverse
to $\fktop \lambda_V$. The above observations enable us to replace 
the linear combination $\sum_i \varphi_{q,i} f_{i i_0 \ldots i_p}$ of functions 
in Proposition \ref{propchechcontrhomcontk} by the values of the singular 
$n$-simplices $\widehat{\sigma}_n (f_{\alpha_0 i_0 \ldots i_p},\ldots
,f_{\alpha_n i_0 \ldots i_p})$ at 
$( \varphi_{q, \alpha_0} , \ldots \varphi_{q, \alpha_n})$ for certain indices 
$\alpha_0,\ldots,\alpha_n \in I$. For this purpose we first observe:

\begin{lemma} \label{hpfcontk}
For any partition of unity $\{ \varphi_{q,i} \mid i \in I\}$ 
subordinate to the open cover $\{ \fktop U_i^{q+1} \mid i \in I \}$ of 
$\fktop \mathfrak{U}[q]$ and $p$-cochain 
$f \in \check{C}^p (\mathfrak{U} , A_{kc}^q)$ 
the maps 
  \begin{equation} \label{predefhpk}
  \fktop U_{i_0 \ldots i_{p-1}}^{q+1} \rightarrow V , \quad 
x \mapsto 
\sigma_n \circ ( f_{\alpha_0 i_0 \ldots i_{p-1}} , \ldots , f_{\alpha_n i_0
  \ldots i_{p-1}} , \varphi_{q, \alpha_0} , \ldots , \varphi_{q, \alpha_n} ) 
(\vec{x})
\end{equation}
-- where for each $x \in X$ the indices $\alpha_0 < \cdots < \alpha_n$ are
those for which $\varphi_{q, \alpha_i}^{-1} (\mathbb{R} \setminus \{ 0 \})$ 
contains $x$ -- are continuous.
\end{lemma}

\begin{proof}
The proof is analogous to that of Lemma \ref{hpfcont}.
\end{proof}

\begin{lemma} \label{hpfcochk}
For any partition of unity $\{ \varphi_{q,i} \mid i \in I\}$ 
subordinate to the open cover $\{ \fktop U_i^{q+1} \mid i \in I \}$ of 
$\fktop \mathfrak{U}[q]$ and $p$-cochain 
$f \in \check{C}^p (\mathfrak{U} , A_{kc}^q)$ the maps defined in 
\ref{predefhp} form a cochain in $\check{C}^{p-1} (\mathfrak{U};A_{kc}^q)$.   
\end{lemma}

\begin{proof}
The proof is analogous to that of Lemma \ref{hpfcoch}.
\end{proof}

\begin{proposition} \label{propchechcontrhomloopk}
For any partition of unity $\{ \varphi_{q,i} \mid i \in I\}$ subordinate to 
the open cover $\{\fktop U_i^{q+1} \mid i \in I \}$ of 
$\fktop \mathfrak{U}[q]$ the homomorphisms 
\begin{eqnarray}\label{eqchechcontrhomloopk}
  h^{p,q} : \check{C}^p ( \mathfrak{U}, A^q) & \rightarrow &  
\check{C}^{p-1} ( \mathfrak{U}, A^q) \\
(h^{p,q} f)_{i_0 \ldots i_{p-1}} & = &  
\sigma_n \circ ( f_{\alpha_0 i_0 \ldots i_{p-1}} , \ldots , f_{\alpha_n i_0 \ldots i_{p-1}} , 
\varphi_{q, \alpha_0} , \ldots , \varphi_{q, \alpha_n} ) \, , \nonumber
\end{eqnarray}
-- where for each $x \in X$ the indices $\alpha_0 < \cdots < \alpha_n$ are
those satisfying $\varphi_{q, \alpha_i} (\vec{x}) \neq 0$ --
form a contraction of the augmented row 
$A^q (\mathfrak{U};V) \hookrightarrow \check{C}^* ( \mathfrak{U}, A^q)$ which
restricts to a row contraction of the augmented sub-complex 
$A_{kc}^q (\mathfrak{U};V) \hookrightarrow \check{C}^* (\mathfrak{U},A_{kc}^q)$.
\end{proposition}

\begin{proof}
The proof is analogous to that of Proposition \ref{propchechcontrhomloopk}.
\end{proof}

\begin{corollary} \label{uqnumicindisoloopk}
For any open covering $\mathfrak{U}=\{ U_i \mid i \in I \}$ of a $k$-space $X$ 
for which the coverings $\{ \fktop U_i^{q+1} \mid i \in I \}$ of the 
$k$-spaces $\fktop \mathfrak{U}[q]$ are numerable the homomorphism 
$i_{kc}^*:A_{kc}^* (\mathfrak{U};V)\rightarrow 
\tot\check{C}^* (\mathfrak{U},A_{kc}^*)$ 
induces an isomorphism in cohomology.
\end{corollary}
Recalling the contractibility condition imposed on the coefficient group $V$ 
we proceed to show:

\begin{theorem}\label{isocontuqnumloopk}
For any loop contractible abelian $k$-group $V$ and open covering 
$\mathfrak{U}$ of a $k$-space $X$ for which each covering 
$\{ \fktop U_i^{q+1} \mid i \in I \}$ of $\fktop \mathfrak{U}[q]$ is numerable 
the inclusion $A_{kc}^* (\mathfrak{U};V) \hookrightarrow A^* (\mathfrak{U};V)$ 
induces an isomorphism in cohomology and the cohomologies 
$\check{H} (\mathfrak{U};V)$, $H_{kc} (\mathfrak{U};V)$ and 
$H (\mathfrak{U};V)$ are isomorphic.
\end{theorem}

\begin{proof}
  The proof is analogous to that of Theorem \ref{isocontuqnum}.
\end{proof}

In this case the \v{C}ech Cohomology $\check{H} (\mathfrak{U};V)$ for the
covering $\mathfrak{U}$ of X can be either computed from the complex 
$A_{kc}^* (\mathfrak{U};V)$ of continuous $\mathfrak{U}$-local cochains or from 
from the complex $A^* (\mathfrak{U};V)$ of $\mathfrak{U}$-local cochains.

\begin{corollary} \label{covbycozloopk}
For any loop contractible abelian $k$-group $V$, generalised partition of 
unity $\{ \varphi_i \mid i \in I\}$ on $X$ and 
$\mathfrak{U}:=\{ \varphi_i^{-1} (R \setminus \{ 0\}) \mid i\in I\}$ 
the inclusion $A_{kc}^* (\mathfrak{U};V) \hookrightarrow A^* (\mathfrak{U};V)$ 
induces an isomorphism in cohomology and the cohomologies 
$\check{H} (\mathfrak{U};V)$, $H_{kc} (\mathfrak{U};V)$ and 
$H (\mathfrak{U};V)$ are isomorphic.
\end{corollary}

\begin{proof}
  The proof is analogous to that of Corollary \ref{covbycoz}.
\end{proof}

Recall that for all metric, locally compact or Hausdorff $k_\omega$-spaces $X$ 
the products $X^{p+1}$ are already compactly Hausdorff generated 
(cf. Lemma \ref{appfinprodofko}); for these kinds of spaces the diagonal 
neighbourhoods of the form $\mathfrak{U}[q]$ are cofinal in all diagonal 
neighbourhoods.  

\begin{corollary}
For metric, locally compact or Hausdorff $k_\omega$-spaces $X$ and loop 
contractible abelian $k$-groups $V$ the \v{C}ech cohomology 
$\check{H} (X;V)$ w.r.t. numerable coverings and the continuous 
Alexander-Spanier cohomology $H_{AS,kc} (X;V)$ w.r.t. numerable coverings are 
isomorphic. 
\end{corollary}

\begin{example}
Real and complex Kac-Moody groups are Hausdorff $k_\omega$-spaces 
(cf.  \cite{GGH06}). Thus for real or complex Kac-Moody groups $G$ and loop
contractible abelian $k$-groups $V$ the cohomologies 
$\check{H} (G;V)$, $H_{AS,kc} (G;V)$ and $H_{AS} (G;V)$ w.r.t. numerable 
coverings are isomorphic.
\end{example}

\begin{corollary}
For any metric, paracompact and locally compact or paracompact Hausdorff
$k_\omega $-space $X$ and loop contractible coefficient $k$-group $V$ the 
\v{C}ech cohomology $\check{H} (X;V)$ and the continuous Alexander-Spanier 
cohomology $H_{AS,c} (X;V)$ are isomorphic. 
\end{corollary}

\begin{example}
If a  metric, paracompact and locally compact or paracompact Hausdorff
$k_\omega $-space $X$ has trivial \v{C}ech cohomology $\check{H} (X;V)$ 
(e.g. if $X$ is contractible) and $V$ is loop contractible, then the
  continuous Alexander-Spanier cohomology $H_{AS} (X;V)$ is trivial as well.
\end{example}

As we did before, we apply these observations to uniform $k$-spaces $X$ with 
open coverings of the form $\mathfrak{U}_U :=\{ U[x] \mid \, x \in X\}$, where 
$U$ is an open entourage of the diagonal in $X \times X$.

\begin{proposition} \label{ubypmloopk}
  If $d: X \times X \rightarrow \mathbb{R}$ is a continuous pseudometric on 
$X$ then for each $\epsilon > 0$, covering 
$\mathfrak{U}=\{ B_d (x,\epsilon) \mid x \in X \}$ of $X$ by open 
$\epsilon$-balls and loop contractible abelian $k$-group $V$ 
the inclusion 
$A_{kc}^* (\mathfrak{U};V) \hookrightarrow A^* (\mathfrak{U};V)$ induces an
isomorphism in cohomology and the cohomologies 
$\check{H} (\mathfrak{U};V)$, $H_{kc} (\mathfrak{U};V)$ and 
$H (\mathfrak{U};V)$ are isomorphic.
\end{proposition}

\begin{proof}
  The proof is analogous to that of Proposition \ref{ubypm}.
\end{proof}

\begin{example}
If $\mathfrak{U}=\{ B (x,\epsilon) \mid x \in X \}$ is an open covering of a 
finite dimensional Riemannian manifold $M$ by open $\epsilon$-balls and 
the coefficient $k$-group $V$ is loop contractible, 
then the cohomology $H_{kc} (\mathfrak{U};V)$ of the complex 
$A_{kc}^* (\mathfrak{U};V)$ is isomorphic to cohomology $H (\mathfrak{U} ;V)$ and 
to the \v{C}ech and singular cohomologies of $M$ (cf. Ex \ref{exugeodconv}). 
If $M$ is an infinite dimensional Riemannian manifold one has to require the 
local existence of geodesics.  
\end{example}

\begin{corollary}
For any open entourage $U$ of a uniform $k$-space $X$ 
and loop contractible abelian $k$-group $V$ the inclusion 
$A_{kc}^* (\mathfrak{U}_U;V)\hookrightarrow A^* (\mathfrak{U}_U;V)$ induces an 
isomorphism in cohomology and the cohomologies 
$\check{H} (\mathfrak{U}_U;V)$, $H_{kc} (\mathfrak{U}_U;V)$ and 
$H (\mathfrak{U}_U;V)$ are isomorphic.
\end{corollary}

For compactly Hausdorff generated topological groups with open coverings of the 
form $\mathfrak{U}_U :=\{ gU \mid \, g \in G\}$, where $U$ is an open identity 
neighbourhood in $G$ we observe:

\begin{corollary} \label{localfortopgrpsandvsvloopk}
For any open identity neighbourhood $U$ of a compactly Hausdorff generated 
topological group $G$ and loop contractible $k$-group $V$ the 
inclusion $A_{kc}^* (\mathfrak{U}_U ;V) \hookrightarrow A^* (\mathfrak{U}_U ;V)$
induces an isomorphism in cohomology and the cohomologies 
$\check{H} (\mathfrak{U}_U;V)$, $H_c (\mathfrak{U}_U;V)$ and 
$H (\mathfrak{U}_U;V)$ are isomorphic.
\end{corollary}

Combining these results with those concerning singular cohomology 
(obtained in Section \ref{seccas}) we observe:

\begin{proposition} \label{propuqnumsingloopk}
For any loop contractible abelian $k$-group $V$ and open covering 
$\mathfrak{U}$ of a $k$-space $X$ for which each set $U_{i_0 \ldots i_p}$ is 
$V$-acyclic and each covering $\{ \fktop U_i^{q+1} \mid i \in I \}$ of 
$\fktop \mathfrak{U}[q]$ is numerable the homomorphism 
$C (\lambda_\mathfrak{U}^* ,V) : 
A_{kc}^* (\mathfrak{U};V) \rightarrow S^* (X,\mathfrak{U};V)$ 
induces an isomorphism $H_{kc} (\mathfrak{U};V) \cong H_{sing} (X;V)$ in 
cohomology and the following diagram is commutative:
\begin{equation*}
  \xymatrix{
H_{kc} (\mathfrak{U};V) \ar[r]^{H (i)}_\cong 
\ar[d]_{H (C (\lambda_\mathfrak{U}^*;V))}^\cong  &
H (\mathfrak{U};V) \ar[r]^{H (j)^{-1} H(i)}_\cong 
\ar[d]_{H (C (\lambda_\mathfrak{U}^* ;V))}^\cong & 
\check{H} (\mathfrak{U};V) \ar@{=}[d] \\
H_{sing} (\mathfrak{U};V) \ar@{=}[r] & 
H_{sing} (\mathfrak{U};V) \ar[r]^\cong_{H (j)^{-1} H(i)} & 
\check{H} (\mathfrak{U};V) 
}
\end{equation*}
In particular the \v{C}ech and the continuous $\mathfrak{U}$-local 
cohomology do not depend on the open cover $\mathfrak{U}$ subject to the 
above conditions chosen.
\end{proposition}

\begin{corollary}
 For loop contractible coefficient $k$-groups $V$ and any open covering 
$\mathfrak{U}$ of a $k$-space $X$ for which each set $U_{i_0 \ldots i_p}$ is 
$V$-acyclic and each covering $\{ \fktop U_i^{q+1} \mid i \in I \}$ of 
$\fktop \mathfrak{U}[q]$ is numerable the singular cohomology 
$H_{sing} (X;V)$ and the \v{C}ech cohomology $\check{H} (\mathfrak{U};V)$ 
for the covering $\mathfrak{U}$ can be computed from the complex 
$A_{kc}^* (\mathfrak{U};V)$ of continuous $\mathfrak{U}$-local cochains. 
\end{corollary}

\begin{example}
For any loop contractible coefficient group $V$ and 'good' cover 
$\mathfrak{U}$ of a topological space $X$ for which 
each covering $\{ U_i^{q+1} \mid i \in I \}$ of $\mathfrak{U}[q]$ is numerable
the morphism $ 
A_{kc}^* (\mathfrak{U};V) \rightarrow S^* (X,\mathfrak{U};V)$ of cochain complexes
induces an isomorphism in cohomology and the cohomologies 
$\check{H} (\mathfrak{U};V)$, $H (\mathfrak{U};V)$, $H_c (\mathfrak{U};V)$ 
and $H_{sing} (X;V)$ are isomorphic.
\end{example}

\begin{lemma}
If $X$ is a $k$-space, $V$ a loop contractible $k$-group and the diagonal 
neighbourhoods $\mathfrak{U}[q]$ for open coverings $\mathfrak{U}$ 
for which the sets $U_{i_0 \ldots i_p}$ are $V$-acyclic and each covering 
$\{ \fktop U_i^{q+1} \mid i\in I\}$ of $\fktop \mathfrak{U}[q]$ 
is numerable are cofinal in all diagonal neighbourhoods, then for each such
covering $\mathfrak{U}$ the cohomology $H_{kc} (\mathfrak{U};V)$ of 
continuous $\mathfrak{U}$-local 
cochains coincides with the continuous Alexander-Spanier cohomology 
$H_{AS,kc} (X;V)$ of $X$. In particular the directed system 
$H_{kc} (\mathfrak{U};V)$ of abelian groups is co-Mittag-Leffler.
\end{lemma}

\begin{proof}
  The proof is analogous to that of Lemma \ref{vacnumarecofin}.
\end{proof}

\begin{example}
If $\mathfrak{U}$ is an open covering of a finite dimensional Riemannian 
manifold $M$ by geodetically convex sets and $V$ is loop contractible, 
then the cohomology of the complex 
$A_{kc}^* (\mathfrak{U};V)$ is isomorphic to the \v{C}ech cohomology $\check{H}
(\mathfrak{U};V)$ and to the singular cohomology $H_{sing} (M;V)$ of $M$. 
If $M$ is an infinite dimensional Riemannian manifold one has to require the 
local existence of geodesics for this argument to be applicable.
\end{example}

\begin{example}
  If $G$ is a Hilbert Lie group, $U$ a geodetically convex identity 
neighbourhood of $G$ and $V$ is loop contractible, 
then the cohomology of the complex 
$A_{kc}^* (\mathfrak{U}_U;V)$ is isomorphic to \v{C}ech cohomology $\check{H}
(\mathfrak{U};V)$ and to the singular cohomology $H_{sing} (G;V)$ of $G$.
\end{example}

As observed before, one can obtain a similar result for locally contractible
compactly Hausdorff generated topological groups without acyclicity condition 
on the open coverings:

\begin{theorem} \label{colimnbhfcontloopk}
For any locally contractible compactly Hausdorff generated group $G$ with 
open neighbourhood filterbase $\mathcal{U}_1$ for which all finite products 
$G^{p+1}$ are $k$-spaces and loop contractible abelian $k$-group $V$ the 
morphisms 
$A_{kc}^* (\mathfrak{U}_U ;V) \hookrightarrow A^* (\mathfrak{U}_U ;V)$ and 
$C (\lambda_{\mathfrak{U}_U}^* ,V) :  A_{kc}^* (\mathfrak{U}_U ; V) 
\rightarrow S^* (\mathfrak{U}_U ; V)$ for all $U \in \mathcal{U}_1$ induce 
isomorphisms  
\begin{equation*}
 \colim_{U \in \mathcal{U}_1} H_{kc} (\mathfrak{U}_U ;V) \cong 
\colim_{U \in \mathcal{U}_1} H (\mathfrak{U}_U ;V) \cong H_{sing} (G;V) \, .
\end{equation*}
in cohomology. 
\end{theorem}

\begin{proof}
  The proof is analogous to that of Theorem \ref{colimnbhfcontloop}.
\end{proof}

\begin{corollary}
For metrisable Lie groups $G$ with open identity neighbourhood filter base 
$\mathcal{U}_1$ and loop contractible $k$-groups $V$ 
the cohomologies $\colim_{U \in \mathcal{U}_1} H_{kc} (\mathfrak{U}_U ;V)$, 
$\colim_{U \in \mathcal{U}_1} H (\mathfrak{U}_U ;V)$, 
$\colim_{U \in \mathcal{U}_1} \check{H} (\mathfrak{U}_U;V)$ and 
$H_{sing} (G;V)$ coincide. 
\end{corollary}

\section{Smoothly Loop Contractible Coefficients}

For an open covering $\mathfrak{U}$ of a (possibly infinite dimensional) 
differential manifold $M$ and abelian Lie groups $V$ one can consider the 
complex $A_s^* (\mathfrak{U};V)$ of smooth $\mathfrak{U}$-local cochains. 
Analogously to the procedure for continuous cochains we impose a condition on
the abelian Lie group $V$:

\begin{definition}
  A (semi-)Lie group $G$ is called \emph{loop contractible}, 
if there exists a smooth contraction $\Phi : G \times I \rightarrow G$ to the 
identity such that $\Phi_t : G \rightarrow G, g \mapsto \Phi (g,t)$ is a 
homomorphism of (semi-)Lie groups for all $t \in I$.
\end{definition}

\begin{example}
  Any topological vector space $V$ is smoothly loop contractible via 
$\Phi (v,t)=t \cdot v$.
\end{example}

\begin{example}
  The path group $PG=C ( (I,\{ 0 \}),(G, \{e \}))$ of based paths of a 
Lie group $G$ is smoothly loop contractible via 
$\Phi_{PG} (\gamma , s) (t):= \gamma (st)$.  
\end{example}

\begin{remark}
  A Lie group $G$ is smoothly loop contractible if and only if the extension 
$\Omega G \hookrightarrow PG \twoheadrightarrow G$ is a semi-direct product: If
$\Phi_G$ is a smooth loop contraction of $G$, then the Lie group homomorphism 
$s:G \rightarrow PG, s (g)(t)=\Phi_G (g,t)$ is a right inverse to the 
evaluation $\mathrm{ev}_1 :PG \rightarrow G$ at $1$; conversely, if such a right
inverse $s$ exists, then the homotopy given by 
$\Phi_G (g,t) := \mathrm{ev}_1 \Phi_{PG} (s(g))(t)$ 
is a smooth loop contraction of $G$.
\end{remark}

Replacing continuity by smoothness, paracompactness by smooth paracompactness 
and the loop contraction by a smooth loop contraction in the previous 
discussion yields:

\begin{proposition}
If $V$ is a smoothly loop contractible abelian Lie group then there exists a 
morphism $\widehat{\sigma} : V^{*+1} \rightarrow C (\varDelta^* , V)$ of 
semi-simplicial Lie groups which is a right inverse to the vertex morphism 
$\lambda_V$ and all the adjoint functions 
$\sigma_n : V^{n+1} \times \varDelta^n \rightarrow V$ are smooth. In addition 
for all $v \in V$ the singular $n$-simplices $\widehat{\sigma}_n (v,\ldots,v)$ 
are the constant maps $\varDelta^n \rightarrow \{ v \}$.
\end{proposition}

\begin{proof}
The morphism $\widehat{\sigma} : V^{*+1} \rightarrow C (\varDelta^* , V)$ is
constructed by smoothing out the construction presented in the last Section. 
\end{proof}

From now on we assume the coefficient Lie group $V$ to be smoothly loop 
contractible with smooth loop contraction $\Phi : V \times I \rightarrow V$, 
which gives rise to a morphism 
$\widehat{\sigma} : V^{*+1} \rightarrow C (\varDelta^* , V)$ of 
semi-simplicial abelian Lie groups that is a right inverse to
$\lambda_V$. Similar to the continuous case we observe:

\begin{lemma} \label{hpfsmooth}
For any smooth partition of unity $\{ \varphi_{q,i} \mid i \in I\}$ 
subordinate to the open cover $\{U_i^{q+1} \mid i \in I \}$ of 
$\mathfrak{U}[q]$ and $p$-cochain $f \in \check{C}^p (\mathfrak{U} , A^q)$ 
the maps 
  \begin{equation} \label{predefhpsmooth}
  U_{i_0 \ldots i_{p-1}}^{q+1} \rightarrow V , \quad 
x \mapsto 
\sigma_n \circ ( f_{\alpha_0 i_0 \ldots i_{p-1}} , \ldots , f_{\alpha_n i_0
  \ldots i_{p-1}} , \varphi_{q, \alpha_0} , \ldots , \varphi_{q, \alpha_n} ) 
(\vec{x})
\end{equation}
-- where for each $x \in X$ the indices $\alpha_0 < \cdots < \alpha_n$ are
those for which $\varphi_{q, \alpha_i}^{-1} (\mathbb{R} \setminus \{ 0 \})$ 
contains $x$ -- are smooth.
\end{lemma}

\begin{proof}
  The proof is analogous to that of Lemma \ref{hpfcont}.
\end{proof}

\begin{lemma}
For any smooth partition of unity $\{ \varphi_{q,i} \mid i \in I\}$ 
subordinate to the open cover $\{U_i^{q+1} \mid i \in I \}$ of 
$\mathfrak{U}[q]$ and $p$-cochain $f \in \check{C}^p (\mathfrak{U} , A_s^q)$ 
the maps defined in \ref{predefhpsmooth} form a cochain in 
$\check{C}^{p-1} (\mathfrak{U};A_s^q)$.   
\end{lemma}

\begin{proof}
  The proof is analogous to that of Lemma \ref{hpfcoch}.
\end{proof}

\begin{proposition} \label{propchechcontrhomloopsmooth}
For any smooth partition of unity $\{ \varphi_{q,i} \mid i \in I\}$ 
subordinate to the open cover 
$\{U_i^{q+1} \mid i \in I \}$ of $\mathfrak{U}[q]$ the homomorphisms 
\begin{eqnarray}\label{eqchechcontrhomloopsmooth}
  h^{p,q} : \check{C}^p ( \mathfrak{U}, A^q) & \rightarrow &  
\check{C}^{p-1} ( \mathfrak{U}, A^q) \\
(h^{p,q} f)_{i_0 \ldots i_{p-1}} & = &  
\sigma_n \circ ( f_{\alpha_0 i_0 \ldots i_{p-1}} , \ldots , f_{\alpha_n i_0 \ldots i_{p-1}} , 
\varphi_{q, \alpha_0} , \ldots , \varphi_{q, \alpha_n} ) \, , \nonumber
\end{eqnarray}
-- where for each $x \in X$ the indices $\alpha_0 < \cdots < \alpha_n$ are
those satisfying $\varphi_{q, \alpha_i} (\vec{x}) \neq 0$ --
form a contraction of the augmented row 
$A^q (\mathfrak{U};V) \hookrightarrow \check{C}^* ( \mathfrak{U}, A^q)$ which
restricts to a row contraction of the augmented sub-complex 
$A_s^q (\mathfrak{U};V) \hookrightarrow \check{C}^* ( \mathfrak{U}, A_s^q)$.
\end{proposition}

\begin{proof}
  The proof is analogous to that of Proposition \ref{propchechcontrhomloop}.
\end{proof}

\begin{corollary} \label{uqnumicindisoloopsmooth}
For any open covering $\mathfrak{U}=\{ U_i \mid i \in I \}$ of a manifold 
$M$ for which the coverings $\{ U_i^{q+1} \mid i \in I \}$ of the 
spaces $\mathfrak{U}[q]$ are smoothly numerable the homomorphism 
$i_c^*:A_s^* (\mathfrak{U};V)\rightarrow \tot\check{C}^* (\mathfrak{U},A_s^*)$ 
induces an isomorphism in cohomology.
\end{corollary}
Recalling the contractibility condition imposed on the abelian Lie group $V$
we proceed to show:

\begin{theorem}\label{isocontuqnumloopsmooth}
For any smoothly loop contractible abelian Lie group $V$ and open covering 
$\mathfrak{U}$ of a manifold $M$ for which each covering 
$\{ U_i^{q+1} \mid i \in I \}$ of $\mathfrak{U}[q]$ is smoothly numerable the 
inclusion $A_s^* (\mathfrak{U};V) \hookrightarrow A^* (\mathfrak{U};V)$ 
induces an isomorphism in cohomology and the cohomologies 
$\check{H} (\mathfrak{U};V)$, $H_s (\mathfrak{U};V)$, $H_c (\mathfrak{U};V)$ 
and $H (\mathfrak{U};V)$ are isomorphic.
\end{theorem}

\begin{proof}
  The proof is analogous to that of Theorem \ref{isocontuqnum}.
\end{proof}

In this case the \v{C}ech Cohomology $\check{H} (\mathfrak{U};V)$ for the
covering $\mathfrak{U}$ of X can be either computed from the complex 
$A_s^* (\mathfrak{U};V)$ of smooth $\mathfrak{U}$-local cochains, from the 
complex $A_c^* (\mathfrak{U};V)$ of continuous $\mathfrak{U}$-local cochains 
or from from the complex $A^* (\mathfrak{U};V)$ of $\mathfrak{U}$-local 
cochains.

In order to obtain results similar to those for continuous cochains we
require the manifold $M$ and their finite products to be smoothly
paracompact. Then passing to the colimit over all smoothly numerable coverings 
yields the classical results: 

\begin{corollary}
  For any manifold $M$ for which all finite powers are smoothly paracompact and 
any smoothly loop contractible abelian Lie group $V$ the 
\v{C}ech cohomology $\check{H} (M;V)$ w.r.t. smoothly numerable coverings 
and the smooth Alexander-Spanier cohomology $H_{AS,s} (X;V)$ w.r.t. numerable 
coverings are isomorphic. 
\end{corollary}

\begin{corollary}
  For any manifold $M$ for which all finite powers are smoothly paracompact 
and any and smoothly loop contractible abelian 
Lie group $V$ the \v{C}ech cohomology $\check{H} (X;V)$ and the smooth 
Alexander-Spanier cohomology $H_{AS,s} (X;V)$ are isomorphic. 
\end{corollary}

\begin{example}
  If a manifold $M$ with smoothly paracompact powers has trivial 
\v{C}ech cohomology 
$\check{H} (M;V)$ (e.g. if $M$ is contractible) and $V$ is smoothly loop 
contractible, then the smooth Alexander-Spanier cohomology $H_{AS,s} (M;V)$ is 
trivial as well.
\end{example}

\begin{proposition} \label{lemappcechcontrhomressmooth}
For any smoothly loop contractible abelian lie group $V$ and open identity 
neighbourhood $U$ of a Lie group $G$ for which all finite products are 
smoothly paracompact, there exists an open identity neighbourhood 
$W \subseteq U$ and homomorphisms 
$h^{p,q} : \check{C}^p ( \mathfrak{U}_U, A^q) \rightarrow  
\check{C}^{p-1} ( \mathfrak{U}_U , A^q)$ satisfying the equation
\begin{equation*} \label{eqappcechcontrhomressmooth}
\res_{\mathfrak{U}_W ,\mathfrak{U}_U}^{p,q} 
\left[ \delta h^{p,q} + h^{p+1,q} \delta \right]
= \res_{\mathfrak{U}_W ,\mathfrak{U}_U}^{p,q} \,
\end{equation*}  
and which leave the sub-rows $\check{C}^* ( \mathfrak{U}_U , A_s^q)$ and 
$\check{C}^* ( \mathfrak{U}_U , A_c^q)$ invariant. In particular the colimit
double complex 
$\colim_{U \in \mathcal{U}_1} \check{C}^* ( \mathfrak{U}_U ;A_s^*)$ is row-exact.
\end{proposition}

\begin{proof}
  For any open identity neighbourhood $U$ of a Lie group $G$ for which all
finite products are smoothly paracompact Lemma \ref{foropenuexistsvandphi} 
shows the existence of an open identity neighbourhood $W \subseteq U$ and 
smooth real-valued functions $\{\varphi_{q,g} \mid G \in G\}$ with locally 
finite supports in $(gU) \times \cdots \times (gU)$ respectively such that 
the restriction of each function $\varphi_q = \sum_{g \in G} \varphi_{q,g}$ to 
$\mathfrak{W}[q]$ is the constant function $1$. Then the homomorphisms 
\begin{eqnarray}
  h^{p,q} : \check{C}^p ( \mathfrak{U}, A^q) & \rightarrow &  
\check{C}^{p-1} ( \mathfrak{U}, A^q) \\
(h^{p,q} f)_{i_0 \ldots i_{p-1}} & = &  
\sigma_n \circ ( f_{\alpha_0 i_0 \ldots i_{p-1}} , \ldots , f_{\alpha_n i_0 \ldots i_{p-1}} , 
\varphi_{q, \alpha_0} , \ldots , \varphi_{q, \alpha_n} ) \, , \nonumber
\end{eqnarray}
-- where for each $x \in X$ the indices $\alpha_0 < \cdots < \alpha_n$ are
those satisfying $\varphi_{q, \alpha_i} (\vec{x}) \neq 0$ -- have the desired 
property.
\end{proof}

\begin{corollary} \label{corfinprodpcthenjsindisoloop}
For any open neighbourhood filterbase $\mathcal{U}_1$ of a Lie group $G$ whose 
finite products are smoothly paracompact and smoothly loop contractible
abelian Lie group $V$ the morphisms $i_s^*$ induce an isomorphism 
$\colim_{U \in \mathcal{U}_1} H_s (\mathfrak{U}_U ;V) \cong 
\colim_{U \in \mathcal{U}_1} H (\tot\check{C}^* (\mathfrak{U}_U ;A_s^*))$.  
\end{corollary}

Summarising the preceding observations for Lie groups we have shown:

\begin{theorem}
 For any open neighbourhood filterbase $\mathcal{U}_1$ of a Lie group $G$ whose 
finite products are smoothly paracompact and smoothly loop contractible abelian
Lie group $V$ the morphisms 
$A_s^* (\mathfrak{U}_U ;V) \hookrightarrow A^* (\mathfrak{U}_U ;V)$ and 
$C (\lambda_{\mathfrak{U}_U}^* ,V) :  A_s^* (\mathfrak{U}_U ; V) 
\rightarrow S^* (\mathfrak{U}_U ;V)$ for all $U \in \mathcal{U}_1$ induce 
isomorphisms  
\begin{equation*}
 \colim_{U \in \mathcal{U}_1} H_s (\mathfrak{U}_U ;V) \cong 
\colim_{U \in \mathcal{U}_1} H (\mathfrak{U}_U ;V) \cong H_{sing} (G;V)
\end{equation*}
in cohomology.
\end{theorem}

\begin{proof}
  The proof is analogous to that of Theorem \ref{colimspclg}.
\end{proof}

\appendix

\section{Partitions of Unity}

\begin{lemma} \label{varphiisummablesubssummable}
  For each summable set of function $\varphi_i :X \rightarrow V$, $i \in I$
  into a complete Hausdorff abelian group $V$ all subsets of functions are
  also summable.
\end{lemma}

\begin{proof}
  Let $\varphi_i :X \rightarrow V$, $i \in I$ be a summable set of functions
  into a complete abelian Hausdorff group $V$ 
and $J \subseteq I$ be a subset of $I$. We claim that for each 
$x \in X$ the net of finite partial sums of $\varphi_j (x) $, $j \in J$ is a 
Cauchy net. For each identity neighbourhood $U$ in $V$ there exists a finite 
subset $I_{x,U}$ of $I$ such that
\begin{equation} \label{sumvarphiiinvarphipu}
 \sum_{i \in I'} \varphi_i (x) \in \varphi (x) + U 
\end{equation}
for all finite supersets $I' \supset I_{x,U} $. Choose an identity
neighbourhood $W$ in $V$ satisfying $W-W \subseteq U$ and consider the finite 
subset 
$J_{x,W}:=J \cap I_{x,W}$ of $J$. For all finite supersets $J',J''$ of
$J_{x,W}$ the above relation \ref{sumvarphiiinvarphipu} implies 
\begin{equation*}
 \sum_{j \in J'} \varphi_j (x) - \sum_{j \in J''} \varphi_j (x) = 
\sum_{j \in J' \cup (I_{x,W} \setminus J )} \varphi_j (x) - 
\sum_{j \in J''\cup (I_{x,W} \setminus J )} \varphi_j (x) \in W -W \subseteq U
\, 
\end{equation*}
hence the the net of finite partial sums of $\varphi_j (x) $, $j \in J$ is a 
Cauchy net. Since $V$ is complete and Hausdorff, this Cauchy net converges and
the limit is unique.
\end{proof}

\begin{lemma}
  For every summable set of real valued functions 
$\varphi_i :X \rightarrow \mathbb{R}$, $i \in I$ the set 
$| \varphi_i |$, $i \in I$ of non-negative functions is also summable.
\end{lemma}

\begin{proof}
  Let $\varphi_i$, $i \in I$ be a summable set of real 
functions. For each point $x \in X$ split the index set $I$ into 
$I_{x,+}:= \{ i \in I \mid \varphi_i (x) \geq 0 \}$ and 
$I_{x,-}:= \{ i \in I \mid \varphi_i (x) < 0 \}$. The sums 
$\sum_{i \in I_{x,+}} \varphi_i (x)$ and $\sum_{i \in I_{x,-}} \varphi_i (x)$
exist by Lemma \ref{varphiisummablesubssummable}, hence $|\varphi_i (x)|$ is
summable with sum  $\sum_{i \in I} |\varphi_i (x)| = \sum_{i \in I_{x,+}}
\varphi_i (x) - \sum_{i \in I_{x,-}} \varphi_i (x)$.
\end{proof}

\begin{lemma} \label{sumvaprhicontthensumabsvalcont}
For every summable set of real valued functions 
$\varphi_i :X \rightarrow \mathbb{R}$, $i \in I$ with continuous sum the sum  
$\sum_{i \in I} | \varphi_i |$ is continuous as well. If 
$\varphi_i$, $i \in I$ is a (generalised) partition of unity, then 
$|\varphi_i| / \sum |\varphi_i|$, $i \in I$ is a non-negative (generalised) 
partition of unity. 
\end{lemma}

\begin{proof}
   Let $\varphi_i : X \rightarrow \mathbb{R}$, $i \in I$ be a summable set of 
real valued functions with continuous sum $\varphi$ and let $\psi$ denote the 
sum of absolute values $|\varphi_i|$. The convergence 
$\sum \varphi_i = \varphi$ means that for all $x \in X$ and $\epsilon > 0$ 
there exists a finite subset $I_{m,\epsilon} \subseteq I$ such that for all 
supersets $I' \supseteq I_{m,\epsilon}$ the inequality 
\begin{equation*}
  \left| \sum_{i \in I'} \varphi_i (x) - \varphi (x) \right| < \epsilon
\end{equation*}
is satisfied.
The set $V_{x,\epsilon} := \{ x' \in X \mid \, 
| \sum_{i \in I_{x,\epsilon }} \varphi_i (x') - \varphi (x')| < \epsilon \}$ is
an open neighbourhood of $x$. For every $x' \in V_{x,\epsilon}$
and finite subset $J \subset I \setminus I_{x,\epsilon}$ the absolute
value of the sum $\sum_{i \in J} \varphi_i (x')$ is less than $2 \epsilon$, 
which implies that the sum $\sum_{i \notin I_{x,\epsilon }} |\varphi_i|$ is
less than $4 \epsilon$. The intersection 
\begin{equation}
W_{x,\epsilon} := V_{x,\epsilon} \cap 
\left( \left[ \sum_{i \in I_{x,\epsilon}} |\varphi_i | \right] -
\psi (x) \right)^{-1} ( (- 4 \epsilon, 4 \epsilon )) 
\end{equation}
is an even smaller open neighbourhood of $x$. 
For all points $x' \in W_{x,\epsilon}$ we observe  
\begin{equation*}
  | \psi (x') - \psi (x) | \leq 
\sum_{i \notin I_{x,\epsilon}} | \varphi_i | (x')  + 
\left| \sum_{i \in I_{x,\epsilon}} | \varphi_i | (x') - \psi (x)
\right| \leq 4 \epsilon + 4 \epsilon = 8 \epsilon \, .
\end{equation*}
Thus for every point $x \in X$ and $\epsilon > 0$ there exists a neighbourhood
$W$ of $x$ such that $\psi (W) \subseteq B_\epsilon ( \psi (x))$, i.e. $\psi$ 
is continuous.
\end{proof}

\begin{lemma} \label{sumcontthensubsumcont}
For each summable set of function $\varphi_i :X \rightarrow V$, $i \in I$ with
continuous sum $\varphi$ into a complete Hausdorff abelian group $V$ the sum
of any subset of functions is also continuous.
\end{lemma}

\begin{proof}
  Let $\varphi_i :X \rightarrow V$, $i \in I$ be a summable set of functions
  into a complete abelian Hausdorff group $V$ with continuous sum 
$\varphi := \sum \varphi_i$ and $J \subseteq I$ be a subset of $I$. We show
that the sum $\varphi_J := \sum_{j \in J} \varphi_j$ is continuous at each
point $x \in X$. For each identity neighbourhood $U$ in $V$ there exists a 
finite subset $I_{x,U}$ of $I$ such that \ref{sumvarphiiinvarphipu} holds for
all supersets $I' \supseteq I_{x,U}$. Furthermore the set 
\begin{equation}
V_{x,U} := \left( \sum_{i \in I_{x,U} \setminus J } \varphi_i - \varphi_i (x) 
\right)^{-1} (U) \cap 
\left( \sum_{i \in I_{x,U}} \varphi_i - \varphi \right)^{-1} (U) \cap 
\cap \left( \varphi - \varphi (x) \right)^{-1} (U)
\end{equation}
is an open neighbourhood of $x$. Choose an identity neighbourhood $W$ in $V$ 
satisfying $W + W + W - W - W\subseteq U$. 
For all points $x' \in V_{x,W}$ and finite supersets $J'$ of 
$J_{x,W}:=I_{x,W} \cap J$ we observe 
  \begin{eqnarray*}
\sum_{j \in J'} \varphi_j (x') - \sum_{j \in J'} \varphi_j (x) & = & 
\sum_{j \in J' \cup (I_{x,W} \setminus J)} \varphi_j (x') - 
\sum_{j \in J' \cup (I_{x,W} \setminus J)} \varphi_j (x) \\
& &  - \sum_{j \in (I_{x,W} \setminus J)} \varphi_j (x') +    
\sum_{j \in (I_{x,W} \setminus J)} \varphi_j (x) \\ 
& \in & ( \varphi (x') + W) - ( \varphi (x) + W ) - W \\
& \subseteq & W + W - W - W
\end{eqnarray*}
Passage to the limit shows that the difference 
$\varphi_J (x') - \varphi_J (x)$ is contained in the closure of $W+W-W-W$,
which in turn is contained in $W + W + W -W -W \subseteq U$. Thus for each
point $x \in X$ and identity neighbourhood $U$ of $V$ there exists a
neighbourhood $V_{x,W}$ of $x$ such that 
$\varphi_J (x') - \varphi_J (x) \in U$ for all $x' \in V_{x,W}$, i.e. the
function $\varphi_J$ is continuous at all points $x \in X$.
\end{proof}

Similar to continuous partitions of unity (as done in \cite{tDTOP}) 
it can be shown that coverings 
by cozero sets of generalised partitions of unity are always smoothly 
numerable. For this purpose we will use the smooth function
\begin{equation}
  f:\mathbb{R} \rightarrow \mathbb{R}, \quad f (x) =
  \begin{cases}
    e^{- \frac{1}{x}} & \text{if }x > 0 \\
0 & \text{if } x \leq 0
  \end{cases}
\end{equation}
to adapt the proof for continuous functions in \cite{tDTOP} to the general
smooth context.

\begin{lemma} \label{genpartthenvarphiin}
  For every generalised smooth partition of unity 
$\{ \varphi_i \mid i \in I \}$ on a manifold $M$ and covering 
by the cozero sets $U_i:=\varphi_i^{-1} (\mathbb{R} \setminus \{ 0 \} )$ 
there exist non-negative smooth real functions 
$\{ \varphi_{i,n} \mid i \in I, \, n \in \mathbb{N} \}$ such that for all 
$n \in \mathbb{N}$  
\begin{enumerate}
\item the collection of supports 
$\{ \supp \varphi_{i,n} \mid i \in I\}$ refines $\{ U_i \mid i \in I \}$,
\item the collection 
$\{ \supp \varphi_{i,n}  \mid i \in I \}$ of
supports is locally finite, 
\end{enumerate}
and such that that for every point $m \in M$ some $\varphi_{i,n}$ satisfies 
$\varphi_{i,n} (m) > 0$ and for fixed $i \in I$ the supports of 
$\varphi_{i,n}$, $n \in \mathbb{N}$ exhaust the open set $U_i$.
\end{lemma}

\begin{proof}
Let $\{ \varphi_i \mid i \in I \}$ be a smooth 
generalised partition of unity on a manifold $M$. We define real valued
functions $\varphi_{i,n}$ on $M$ via
\begin{equation*}
  \varphi_{i,n} := f \circ \left( \varphi_{i}^2 - \frac{1}{(n+1)^2} \right) \, .
\end{equation*}
The support of $\varphi_{i,n}$ is contained in the cozero set 
$\varphi_{i,n+1}^{-1} ((0,\infty))$ and for fixed $i \in I$ the supports of 
$\varphi_{i,n} \, n \in \mathbb{N}$ exhaust $U_i$ by construction; 
this proves (1) and the last claim. 
Furthermore for every point $m \in M$ the convergence 
$\sum \varphi_i (m)= 1$ means that for all
$\epsilon > 0$ there exists a finite subset $I_{m,\epsilon} \subseteq I$ such
that for all supersets $I' \supseteq I_{m,\epsilon}$ the inequality 
\begin{equation*}
  \left| \sum_{i \in I'} \varphi_i (m) -1 \right| < \epsilon
\end{equation*}
is satisfied. For every $n$  
the set 
$V_{m,n} :=\{ m' \in M \mid \, |\sum_{i \in I_{m,1 / 2n}} (m')-1|< 1/2n \}$ 
is an open neighbourhood of $m$ on which
  all functions $\varphi_i$, $i \notin I_{m,1 / n}$ satisfy 
$|\varphi_i | <  1 / n$. This implies $\varphi_i^2 \leq 1 / n^2$
and $\varphi_{i,n} =0$ on $V_{m,n}$ for all $i$ which are not contained in the
finite set $I_{m, 1 / 2n}$. Therefore the collection 
$\{ \varphi^{-1}_{i,n+1}  ( (0,\infty) ) \mid i \in I \}$ of cozero sets and
the collection 
$\{ \supp \varphi_{i,n} \mid i \in I \}$ of
supports are locally finite, which proves (2). 
Finally, for each point $m
\in M$ there exists some $\varphi_i$ satisfying $\varphi_i (m) \neq 0$. For
all $n \in \mathbb{N}$ satisfying $1 / n < | \varphi_i (m) |$ the
function $\varphi_{i,n}$ also satisfies $\varphi_{i,n} (m) \neq 0$.
\end{proof}

\begin{proposition} \label{uisnumiffvarphiin}
  An open covering $\mathfrak{U}$ of a smooth manifold $M$ is smoothly 
numerable if and only if there exist non-negative smooth real functions 
$\varphi_{i,n}$, $i \in I, \, n \in \mathbb{N}$ such that for all 
$n \in \mathbb{N}$  
\begin{enumerate}
\item the collection of supports $\{ \supp \varphi_{i,n} \mid i \in I\}$ 
refines $\mathfrak{U}$, 
\item the collection 
$\{ \supp \varphi_{i,n} \mid i \in I \}$ of supports is locally finite, 
\end{enumerate}
and such that that for every point $m \in M$ some $\varphi_{i,n}$ satisfies 
$\varphi_{i,n} (m) > 0$.
\end{proposition}

\begin{proof}
The proof is an adaption of the proof of \cite[Lemmata 5.5, 4.6]{tDTOP} to 
the general smooth context. The forward implication is proved in Lemma 
\ref{genpartthenvarphiin}, so only the backward implication requires proof. Let 
$\varphi_{i,n}$, $i \in I, \, n \in \mathbb{N}$ be smooth real valued
functions with the above properties. Replacing $\varphi_{in,n}$ with 
$\varphi_{i,n}^2 / (1 + \varphi_{i,n}^2)$ we can w.l.o.g. assume that the 
functions $\varphi_{i,n}$ take values in the uni interval. 
For each $n \in \mathbb{N}$ the collection 
$\{ \supp \varphi_{i,k} \mid i \in I ,\, k < n \}$ of supports is locally 
finite, hence the sum
\begin{equation*}
 q_n := \sum_{i \in I, k < n} \varphi_{i,k}
\end{equation*}
(where $q_0 =0$) is a smooth real valued function on $M$. 
As a consequence the composition 
$\psi_{i,n}:= f \circ (\varphi_{i,n} - n \cdot q_n)$ is smooth as well. We claim
that the collection 
$\{ \psi_{i,n}^{-1} (0,\infty)  \mid i \in I ,\, n \in \mathbb{N} \}$ of
cozero sets an open covering of $M$ and that the supports 
$\{ \supp \psi_{i,n} \mid i \in I ,\, n \in \mathbb{N} \}$ form a 
locally finite covering of $M$ which refines the open covering $\mathfrak{U}$. 
If $n$ is minimal such there exists a function $\varphi_{i,n}$ satisfying 
$\varphi_{i,n} (m) \neq 0$, then $q_n (m) =0$ and 
$\psi_{i,n} (m)= \varphi_{i,n} -0 \neq 0$. Thus for each point $m \in M$ there
exists an index $(i,n)$ such that $\psi_{i,n} (m) \neq 0$, i.e. the collection 
$\{ \psi^{-1}_{i,n}  ( (0,\infty) ) \mid i \in I ,\, n \in \mathbb{N} \}$ of
cozero sets is an open covering of $M$. Moreover the sequence of functions
$q_n$ is monotone increasing, so for each $m \in M$ there exists $N \in
\mathbb{N}$ such that $N \cdot q_N (m) > 1$ hence also $N \cdot q_N (m) > 1$ 
in a neighbourhood $V_m$ of $m$. 
As a consequence the supports of all the functions $\psi_{i,n}$ with $n > N$ 
do not intersect $V_m$. So the collection 
$\{ \supp \psi_{i,n}  \mid i \in I ,\, n \in \mathbb{N} \}$ of supports is 
locally finite. Normalisation of 
$\{ \psi_{i,n} \mid i \in I, n \in \mathbb{N} \}$ yields a smooth partition of 
unity subordinate to $\mathfrak{U}$.
\end{proof}

\begin{theorem}
Every covering of a manifold $M$ by the cozero sets of a generalised partition 
of unity $\{ \varphi_i \mid i \in I \}$ is numerable. 
\end{theorem}

\begin{proof}
  Let $\{ \varphi_i \mid i \in I \}$ be a generalised smooth partition of
  unity on a manifold $M$ and construct non-negative functions $\varphi_{i,n}$ 
as in the proof of Lemma \ref{genpartthenvarphiin} and further a partition of 
unity $\{ \psi_{i,n} \mid i \in I, \, n \in \mathbb{N} \}$ as in Proposition
\ref{uisnumiffvarphiin}. Then $\{ \psi_i \mid i \in I, \, n \in \mathbb{N} \}$ 
is a partition of unity subordinate to 
$\{ \varphi_i^{-1} ( \mathbb{R} \setminus 0 ) \mid i \in I \}$. 
\end{proof}

\section{$k$-Spaces}

\begin{definition}
A continuous function from a compact Hausdorff space into an arbitrary 
topological space $X$ is called a \emph{probe over $X$}.
\end{definition}

\begin{definition}
  The \emph{$k$-topology} on a topological space $X$ is the final topology of 
all probes over X. The underlying set of $X$ equipped with the $k$-topology is 
denoted by $\fktop X$.
\end{definition}

\begin{lemma}
  The  $k$-topology on a topological space $X$ is finer
  than the original topology of $X$, i.e. the set theoretic identity map 
$\fktop X \rightarrow X$ is continuous.
\end{lemma}

\begin{definition}
  A topological space is called a $k$-spaces if $\fktop X=X$. The full 
subcategory of $\tops$ with objects all $k$-spaces is denoted by $\ktop$. 
\end{definition}

Any continuous function $f:X \rightarrow Y$ between topological space $X$ and
$Y$ gives rise to a continuous function 
$\fktop (f):\fktop X \rightarrow \fktop Y$, which coincides with 
$f$ on the underlying set. These assignments constitute a functor 
\begin{equation*}
  \fktop : \tops \rightarrow \ktop.
\end{equation*}

\begin{proposition} \label{ktopiscrh}
The category $\ktop$ is a coreflective subcategory of $\tops$ with coreflector 
$\fktop$, i.e. the functor $\fktop$ is a right adjoint to the inclusion 
$\ktop \rightarrow \tops$.
\end{proposition}

\begin{proof}
Let $X$ be a $k$-space and $Y$ be an arbitrary topological space. The 
continuous function $\epsilon_Y :\fktop Y \rightarrow Y$ induces an injective 
map
 \begin{equation*}
    i_* : \hom_\ktop (X,\fktop Y) \rightarrow \hom_\tops (X,Y)
  \end{equation*} 
which is natural in $X$ and $Y$. It remains to show that $i_*$ is surjective. 
Let $f:X \rightarrow Y$ be a continuous function. Then the function 
$\fktop (f): X = \fktop X \rightarrow \fktop Y$ is continuous and 
$i_* \fktop (f)=f$. Therefore $i_*$ is bijective.
\end{proof}

\begin{corollary} \label{colinktoptops}
The inclusion $\ktop \rightarrow \tops$ is cocontinuous, i.e. the colimits of 
$k$-spaces in $\ktop$ coincide with those in $\tops$.
\end{corollary}

\begin{corollary} \label{corefliscont}
  The coreflector $\fktop :\tops \rightarrow \ktop$ is continuous, i.e. it 
preserves limits.
\end{corollary}

\begin{proposition} \label{prodinktop}
  The product of spaces $Y_i$ in $\ktop$ is given by 
$\fktop \left( \prod Y_i \right)$, 
where $\prod Y_i$ is the product in $\tops$.
\end{proposition}

\begin{proof}
  By the use of Lemma \ref{ktopiscrh} one obtains for every $k$-space $X$ the 
following chain of natural isomorphisms:
  \begin{eqnarray*}
    \hom_\ktop \left( X , \fktop \prod Y_i \right) & = & 
\hom_\tops \left( X ,\prod Y_i \right) \\ 
& \cong & 
\prod \hom_\tops (X,Y_i) \cong \prod \hom_\ktop (X, Y_i),
  \end{eqnarray*}
which is what was to be proved.
\end{proof}

\begin{theorem}  \label{closedssofksp}
closed subspaces of $k$-spaces are $k$-spaces.
\end{theorem}

\begin{proof}
  Let $X$ be a $k$-space and $A \subset X$ be a closed subspace. It is to show 
that $\fktop A \rightarrow A$ is an isomorphism, i.e. every subset $B$ of $A$ 
that is closed in $\ktop A$ is already closed in $A$. Let $B$ be such a closed 
subset of $\fktop A$. 
Since $A$ is closed in $X$, the inverse image $p^{-1} (A)$ under 
any probe $p:C \rightarrow X$ is closed in the compact Hausdorff space $C$,
hence the inverse image $p^{-1} (A)$ also is compact. The restriction 
$p_{\mid p^{-1} (A)}$ of $p$ to the compact subspace $p^{-1} (A)$ of $C$ is a 
probe on $A$. The inverse image $p^{-1} (B)$ is equal to the inverse image 
$p_{\mid p^{-1} (A)}^{-1} (B)$ under the probe $p_{\mid p^{-1} (A)}$. 
The latter inverse image is closed in $p^{-1} (A)$ by assumption.
Being a closed subspace of a closed subspace $p^{-1} (A)$ of $C$ it is closed 
in $C$. Thus the inverse image of $B$ under any probe $p:C \rightarrow X$ is 
closed in $C$, i.e. $B$ is closed in $X$.
\end{proof}

\begin{lemma} \label{fktoppresclemb}
  The coreflector $\fktop : \tops \rightarrow \ktop$ preserves closed 
embeddings.
\end{lemma}

\begin{proof}
Let $A$ be a closed subspace of a topological space $X$ and let 
$i_A:A \hookrightarrow X$ denote the inclusion. To show that 
$\fktop (i_A):\fktop A \rightarrow \fktop X$ is closed, it suffices to prove
that the image $\fktop (i_A) (B)$ of every closed subset $B$ of $\fktop A$ is
closed in $\fktop X$. If $B \subseteq \fktop A$ is such a closed subset, then
every inverse image ${f'}^{-1} (B)$ under a probe $f':C' \rightarrow A$ from a
compact Hausdorff space $C'$ into $A$ is closed. 
If $f:C \rightarrow X$ is a probe, then the inverse image $C':=f^{-1} (A)$ is 
closed in $C$, hence a compact subspace of $C$. Therefore the 
restriction and corestriction $f_{\mid C'}^{\mid A}$ is a probe into $A$; 
The inverse image $f^{-1} (B)$ of $B$ under the probe $f:C \rightarrow X$
coincides with the inverse image ${f_{\mid C'}^{\mid A}}^{-1} (B)$ of $B$
under the probe $f_{\mid C'}^{\mid A}$, which is closed in $C'$. 
Because $C'$ is a closed subspace of $C$, the inverse image $f^{-1} (B)$ also 
is closed in $C$. Thus the inverse image $f^{-1} (B)$ of $B$
under any probe $f:C \rightarrow X$ is closed, hence $B$ is closed in 
$\fktop X$.
\end{proof}

\begin{lemma}
  Any open subspace of a k-space is a k-space.
\end{lemma}

\begin{proof}
  See \cite[Satz 6.6]{tDTOP} for a proof.
\end{proof}

\begin{lemma} \label{fktoppresopemb}
  The coreflector $\fktop : \tops \rightarrow \ktop$ preserves open embeddings.
\end{lemma}

\begin{proof}
  Let $U$ be an open subspace of a topological space $X$ and let $j_U : U
  \hookrightarrow X$ denote the inclusion. For the natural transformation
  $\epsilon : \fktop \rightarrow \id_\tops$ one obtains the commutative
  diagram 
  \begin{equation*}
    \xymatrix{
\fktop U \ar[d]_{\epsilon_U} \ar[r]^{\fktop (j_U)} & \fktop X \ar[d]^{\epsilon_X} \\
U \ar[r]^{j_U} & X 
}
  \end{equation*}
The open subspace $\epsilon_X^{-1} (U)$ of the $k$-space $\fktop X$ is a
$k$-space itself. Restricting the morphism $\epsilon_X$ to the $k$-space 
$\epsilon_X^{-1} (U)$ of $\fktop X$ and corestricting it to $U$ we obtain the 
equality 
$\id_U = {\epsilon_X}_{\mid \epsilon_X^{-1} (U)}^{\mid U} \circ \fktop (j_U)$,
i.e. $\fktop U$ is homeomorphic to $\epsilon_X^{-1} (U)$ and $\fktop (j_U)$ is
the inclusion of an open subspace.
\end{proof}

\section{$k_\omega$-Spaces} \label{appkotops}

There exists a variant of the category $\ktop$ of $k$-spaces, which shares
many of the properties of the category $\ktop$. It is formed by the class of
all topological spaces whose topology is the weak topology with respect to
some countable set of compact Hausdorff spaces:

\begin{definition}
  A topological space $X$ is called a \emph{$k_\omega$-space} if it has the 
weak topology w.r.t. a countable set $\{ f_n (K_n) \}$ of images of compact 
Hausdorff spaces $K_n$ under continuous functions $f_n:K_n \rightarrow X$.
The full subcategory of $\tops$ consisting with objects all 
$k_\omega$-spaces is denoted by $\kotop$.
\end{definition}
Since the images $f_n (K_n)$ in the above definition are compact subspaces of
the topological space $X$, this implies that $X$ has the 
colimit topology of the ascending sequence $\bigcup_{i=1}^n K_i$ of compact
subspaces of $X$. Conversely, if $X$ is the direct limit of an ascending
sequence of compact subspaces, then it is a $k_\omega$-space.
Similar to the class of $k$-spaces the class of $k_\omega$-spaces can be 
described as quotients of compact Hausdorff-spaces:

\begin{lemma} \label{appcharofko}
  A topological space is a $k_\omega$-space if and only if it is the quotient 
of a disjoint union of at most countably many compact Hausdorff spaces.
\end{lemma}

\begin{proof}
  The proof of \cite[Theorem XI.9.4]{topdug} generalises to $k_\omega$-spaces.
\end{proof}

\begin{lemma} \label{appquotofkoareko}
  Quotients of $k_\omega$-spaces are $k_\omega$-spaces.
\end{lemma}

\begin{proof}
Since compositions of quotient maps are quotient maps, this is a consequence 
of the characterisation of $k_\omega$-spaces in Lemma \ref{appcharofko}.
\end{proof}

\begin{lemma} \label{appcduofkoareko}
  Finite and countable disjoint unions of $k_\omega$-spaces are 
$k_\omega$-spaces.
\end{lemma}

\begin{proof}
  This follows from the characterisation of $k_\omega$-spaces in Lemma 
\ref{appcharofko}.
\end{proof}
Summarising the last two lemmata we observe that the subcategory $\kotop$ of 
$\tops$ has the same countable colimits:

\begin{proposition} \label{ccolimofkoareko}
  Finite and countable colimits of $k_\omega$-spaces in $\tops$ are 
$k_\omega$-spaces.
\end{proposition}

\begin{definition}
  A topological space $X$ is called \emph{locally $k_\omega$} if every point
  has a neighbourhood filter basis of $k_\omega$-spaces.
\end{definition}

\begin{lemma} \label{appfinprodofko}
  Finite products of Hausdorff $k_\omega$-spaces are Hausdorff $k_\omega$-spaces.
\end{lemma}

\begin{proof}
This is a special case of \cite[Lemma 1.1 (b)]{GGH06}, cf. 
\cite[Proposition 4.2 (c)]{GGH06}.
\end{proof}

\bibliographystyle{amsalpha}
\bibliography{CohomologyOfLocalCochains}

\end{document}